\documentclass[reqno]{amsart}
\usepackage{amsthm}
\usepackage{amssymb}
\usepackage{mathabx}
\usepackage{bbold}

\newtheorem{theorem}{Theorem}[section]
\newtheorem{lemma}[theorem]{Lemma}
\newtheorem{proposition}[theorem]{Proposition}
\theoremstyle{definition}
\newtheorem*{definition}{Definition}
\theoremstyle{remark}
\newtheorem{observation}{Remark}
\newcommand{\R}{\mathbb R}
\newcommand{\dd}{\,d}
\newcommand{\les}{\lesssim}
\newcommand{\be}{\begin{equation}}
\newcommand{\ee}{\end{equation}}
\newcommand{\lb}{\label}
\newcommand{\mc}{\mathcal}
\DeclareMathOperator{\sgn}{sgn}
\DeclareMathOperator{\re}{Re}
\DeclareMathOperator{\im}{Im}
\DeclareMathOperator{\supp}{supp}
\newcommand{\U}{\mc U}
\newcommand{\B}{\mc B}
\newcommand{\one}{{\mathbb 1}}
\numberwithin{equation}{section}

\title{Decay estimates for the wave equation in two dimensions}
\author[M.\ Beceanu]{Marius Beceanu}
\address{UC Berkeley Mathematics Department, Berkeley, CA 94720}
\email{mbeceanu@berkeley.edu}
\date{}
\subjclass[2010]{35L05; 35B34, 35L71, 35B45}
\keywords{wave equation, evolution equation, Hamiltonian equation, dimension two, dispersive estimates, reversed Strichartz estimates, pointwise decay estimates, weighted estimates, Strichartz estimates}

\begin{document}
\maketitle
\begin{abstract} We establish Strichartz estimates (both reversed and some direct ones), pointwise decay estimates, and weighted decay estimates for the linear wave equation in dimension two with an almost scaling-critical potential, in the case when there is no resonance or eigenvalue at the edge of the spectrum.

We also prove some simple nonlinear applications.
\end{abstract}

\section{Introduction}
\subsection{Results} Consider the linear wave equation with a real-valued scalar potential in dimension two:
\be\lb{eq_wave}
f_{tt} - \Delta f + V f = F, f(0)=f_0, f_t(0)=f_1.
\ee

A natural condition for equation (\ref{eq_wave}) to be well-posed is that $H=-\Delta+V$ should be self-adjoint. The existence of a self-adjoint extension was shown in \cite{simon} under the assumption that
$$
\lim_{\epsilon \to 0} \sup_y \int_{|x-y|<\epsilon} |V(x)| \log_-|x-y|  = 0.
$$
If $H=-\Delta+V$ is self-adjoint, then the solution is given by
$$
f(t) = \cos(t\sqrt H) f_0 + \frac {\sin(t\sqrt H)}{\sqrt H} f_1 + \int_0^t \frac {\sin((t-s)\sqrt H)}{\sqrt H} F(s) \dd s.
$$
The following quantity, called energy, is constant as a function of $t$ and remains bounded for all time if it is initially finite:
$$
E[f](t) := \int_{\R^2 \times \{t\}} f_t^2 + |\nabla f|^2 + V f^2 \dd x.
$$

Under rather general assumptions, the spectrum of the Hamiltonian $H$ consists of  the absolutely continuous component $[0, \infty)$ and the possibly empty point spectrum, containing negative energy eigenstates and zero energy eigenfunctions or resonances (see \cite{ioje} and \cite{kota} concerning the absence of positive eigenvalues).

The solution's projection on the point spectrum of $H$ lacks any decay and may even have exponential growth. Thus, in order to obtain dispersive estimates, we must first project away from the point spectrum. Zero energy states pose an even more serious obstruction: even after projecting them away, dispersion may only take place at a suboptimal rate or not at all. This is why we assume the absence of zero energy eigenfunctions and resonances in this paper.

Since the free ($V=0$) linear equation (\ref{eq_wave}) presents a resonance at energy zero in dimension two, some estimates obtained in the presence of a potential may be better than in the free case, as long as the potential eliminates the zero energy resonance. This is the case for inequalities (\ref{c1}), (\ref{c3}), (\ref{c8}), (\ref{c9}), (\ref{d1}), and (\ref{d2}) in this paper.

In the generic case when there is no resonance or eigenvalue at the edge of the spectrum of $H$, we establish several estimates for the projection on the continuous spectrum of the solution to equation (\ref{eq_wave}). We prove pointwise $t^{-1/2}$ decay estimates, weighted integrable-in-time decay estimates, reversed Strichartz estimates, and some ordinary Strichartz estimates. All of them take place under almost scaling-invariant (hence optimal) decay conditions on the potential $V$.

One novelty of our results is that we prove reversed Strichartz estimates for the wave equation in dimension two, that is estimates that hold in the reversed Strichartz norms
$$
\|f\|_{L^q_x L^r_t} := \left(\int_{\R^2} \|f(x, \cdot)\|_{L^r_t}^q \dd x\right)^{\frac 1 q}.
$$
Such estimates have applications in the study of semilinear equations, see for example \cite{beceanu2} or Proposition \ref{prop_semilinear}. Another application is proving the pointwise convergence of the solution of the wave equation to the initial data, see \cite{becgol}.

Although there is some overlap between the range of allowed exponents for reversed Strichartz estimates and the ones for ordinary Strichartz estimates, neither of them is contained in the other. Indeed, in dimension two, reversed Strichartz estimates hold for $(\frac 1 q, \frac 1 r)$ inside the rectangle with vertices $(\frac 1 8, \frac 1 2)$, $(0, \frac 1 2)$, $(\frac 1 8, 0)$, and $(0, 0)$, while ordinary Strichartz estimates (see \cite{keel}) hold inside the right triangle with vertices $(\frac 1 2, 0)$, $(0, \frac 1 4)$, and $(0, 0)$. For comparison, Strichartz estimates with radial data (see \cite{fang} or \cite{sterbenz}) hold inside the right triangle with vertices $(\frac 1 2, 0)$, $(0, \frac 1 2)$, and $(0, 0)$.

In stating the theorems below, we assume that $0$ is a regular point of the spectrum of $-\Delta+V$. This property is defined by the Definition on p.\ 12. Let the spaces $\mc K_\theta$ and $\mc K^*_\theta$ be defined by (\ref{local_kato}) and (\ref{kato*}) respectively. We also denote by $L^{p, q}$ the usual Lorentz spaces, see \cite{bergh} for their definition and properties.

Our main result is as follows:

\begin{theorem}\lb{dispersive} Assume that $(1+\log_+|x|)^2 V \in L^1_x$, for some $q \in (1, \infty]$ $V \in L^q_{loc}$ and $\lim_{R \to \infty} R^{1-1/q} \|V\|_{L^q(|x|\in [R, 2R])}=0$, and $0$ is a regular point of the spectrum of $H=-\Delta+V$. Let $P_c$ be the projection on the continuous spectrum of $H$. Then
\be\lb{c1}
\int_1^\infty \left|\frac{\sin(t\sqrt H) P_c}{\sqrt H}\right|(x, y) \dd t \les (1+\log_+|x|)(1+\log_+|y|),
\ee
\be\lb{c2}
\int_0^1 \left|\frac{\sin(t\sqrt H) P_c}{\sqrt H} f\right|(x) \dd t \les \|f\|_{L^1_x \cap \mc K_1},
\ee
and for $p \in [1, \infty]$
\be\lb{c3}
\left\|\int_{-\infty}^t \frac{\sin((t-s)\sqrt H) P_c}{\sqrt H} F(s) \dd s\right\|_{(1+\log_+|x|) L^\infty_x L^p_t} \les \|F\|_{((1+\log_+|x|)^{-1} L^1_x \cap \mc K_1) L^p_s}.
\ee
In addition,
\be\lb{c40}
\int_0^1 |\cos(t\sqrt H) P_c f|(x) \dd t \les \|\nabla f\|_{L^1_x},
\ee
\be\lb{c4}
(1+\log_+|x|)^{-1} \int_0^\infty|\cos(t\sqrt H) P_c f|(x) \dd t \les \|\nabla f\|_{L^1_x},
\ee
and for $1 \leq p \leq \infty$
\be\lb{c5}
\left\|\int_{-\infty}^t \cos((t-s)\sqrt H) P_c F(s) \dd s\right\|_{(1+\log_+|x|)L^\infty_x L^p_t} \les \|\nabla F\|_{L^1_x L^p_t}.
\ee
Furthermore, for $1 \leq q_1, q_2, r_1, r_2, \sigma, \tilde \sigma \leq \infty$, one has that
\be\lb{c6}
\left\|\int_{-\infty}^t \frac {\sin((t-s)\sqrt H) P_c}{\sqrt H} F(x, s) \dd s\right\|_{L^{q_1, \sigma}_xL^{r_1, \tilde \sigma}_t} \les \|F\|_{L^{q_2, \sigma}_x L^{r_2, \tilde \sigma}_s},
\ee
\be\lb{c7}
\left\|\int_{-\infty}^t \cos((t-s)\sqrt H) P_c F(x, s) \dd s\right\|_{L^{q_1, \sigma}_xL^{r_1, \tilde \sigma}_t} \les \|\nabla F\|_{L^{q_2, \sigma}_x L^{r_2, \tilde \sigma}_s},
\ee
where $\frac 2 {q_1} + \frac 1 {r_1} + 2 = \frac 2 {q_2} + \frac 1 {r_2}$ and $0 < \frac 1 {r_2} - \frac 1 {r_1} \leq \frac 1 2$. When $r_1=\infty$ $L^{r_1, \tilde \sigma} = L^\infty$. When $r_2=1$ $L^{r_2, \tilde \sigma} = L^1$. When $r_1=\infty$, $r_2=2$  then $L^{r_1, \tilde \sigma}=L^\infty$ and $L^{r_2, \tilde \sigma}=L^{2, 1}$. When $r_1 =2$, $r_2=1$ then $L^{r_1, \tilde \sigma}=L^{2, \infty}$ and $L^{r_2, \tilde \sigma} = L^1$. When $q_1=\infty$ $L^{q_1, \sigma}=L^\infty$ and $L^{q_2, \sigma}=L^{q_2, 1}$ and when $q_2=1$ $L^{q_1, \sigma}=L^{q_1, \infty}$ and $L^{q_2, \sigma}=L^1$.

Next,
\be\lb{c8}
\left\|\frac{\sin(t\sqrt H)P_c}{\sqrt H} f\right\|_{((1+\log_+|x|)L^\infty_x+\mc K_1^*) L^\infty_t} \les \|f\|_{L^2_x}
\ee
and
\footnote{Note that the following more refined estimate is also true, but too cumbersome to prove in this paper: $\left\|\frac{\sin(t\sqrt H)P_c}{\sqrt H} f\right\|_{(L^\infty_x+\mc K_1^*) L^\infty_t([0, 1])} + \left\|\frac{\sin(t\sqrt H)P_c}{\sqrt H} f\right\|_{(1+\log_+|x|)L^\infty_x L^\infty_t([1, \infty))} \les \|f\|_{L^2_x}$. The same applies to the subsequent cosine estimate (\ref{c9}).
}
\be\lb{c9}
\|\cos(t\sqrt H)P_c f\|_{((1+\log_+|x|)L^\infty_x+\mc K_1^*) L^\infty_t} \les \|f\|_{H^1_x}.
\ee
The dual estimates are also true.

Moreover, assume $V \in L^q_x$. For $\frac 1 4<s<1$ and $\frac 2 q + \frac 1 r = 1-s$, $8 \leq q < \infty$ (and including $L^{q, 2}=L^\infty$ in the case of the free Laplacian, but excluding $(q, r)=(\infty, 2)$), $2 \leq r \leq \infty$ (where $L^{r, 2} = L^\infty$ when $r=\infty$), excluding the endpoint $(q, r, s) = (8, \infty, \frac 3 4)$, we have
\be\lb{c11}
\left\|\frac{\sin(t\sqrt H)P_c}{\sqrt H} f\right\|_{L^{q, 2}_x L^{r, 2}_t} \les \||H|^{\frac{s-1} 2} f\|_{L^2_x}
\ee
(as per Lemma \ref{norm_equivalence}, this is dominated by $\|f\|_{\dot H^{s-1+}_x\cap\dot H^{s-1-}_x}$) and
\be\lb{c12}
\|\cos(t\sqrt H)P_c f\|_{L^{q, 2}_x L^{r, 2}_t} \les \||H|^{s/2} f\|_{L^2_x} \les \|f\|_{H^s_x}.
\ee
With $q'$ and $r'$ dual exponents to the ones above we have by duality
\be\lb{c13}
\left\||H|^{\frac{1-s}2} \int_{-\infty}^\infty \frac {\sin(t\sqrt H)P_c}{\sqrt H} F(t)\right\|_{L^2_x} \les \|F\|_{L^{q', 2}_x L^{r', 2}_t}
\ee
(the left-hand side norm dominates $\|\cdot\|_{\dot H^{1-s+}_x+\dot H^{1-s-}_x}$) and
\be\lb{c14}
\left\|\int_{-\infty}^\infty \cos(t\sqrt H)P_c F(t) \dd t\right\|_{H^{-s}_x} \les \|F\|_{L^{q', 2}_x L^{r', 2}_t}.
\ee
\end{theorem}
Note that by \cite{oh} the following endpoint inequality is true:
$$
\|e^{it\sqrt{-\Delta}} f\|_{L^\infty_x L^2_t} \les \|f\|_{\dot B^{1/2}_{2, 1}}.
$$

\begin{observation} The condition $V \in L^q_{loc}$, $\lim_{R \to \infty} R^{1-1/q} \|V\|_{L^q(|x|\in [R, 2R])}=0$, can be replaced by requiring that $V \in \mc K_1$ and
$$
\lim_{\epsilon \to 0} \sup_y \int_{|x-y|<\epsilon} |V(x)| \log_-|x-y| \dd x = 0,
$$
provided that a spectral condition is satisfied: $U+T(\lambda)$ is invertible in $L^2$ for all $\lambda>0$ (see Lemma \ref{invertibility}).

Note that this allows for potentials that are measures with singular support. For an example of dispersive estimates with such potentials (in dimension three), see \cite{goldberg}.
\end{observation}


In addition to the reversed Strichartz estimates of Theorem \ref{dispersive}, we can easily prove several direct Strichartz estimates.

\begin{proposition} Consider a potential $V$ such that $(1+\log_+|x|)^2 V \in L^1_x$, for some $q \in (1, \infty]$ $V \in L^q_{loc}$ and $\lim_{R \to \infty} R^{1-1/q} \|V\|_{L^q(|x|\in [R, 2R])}=0$, and $0$ is a regular point of the spectrum of $H=-\Delta+V$. Then for $(\frac 1 q, \frac 1 r)$ contained in the triangle with vertices $(\frac 1 2, 0)$, $(\frac 1 8, \frac 1 8)$, and $(0, 0)$, excluding the endpoint $(0, 0)$, and for $\frac 2 q  + \frac 1 r = 1-s$
$$
\left\|\frac{\sin(t\sqrt H)P_c}{\sqrt H} f \right\|_{L^r_t L^q_x} \les \|H^{\frac{1-s} 2} f\|_{L^2_x}
$$
(as per Lemma \ref{norm_equivalence}, this is dominated by $\|f\|_{\dot H^{s-1+}_x\cap \dot H^{s-1-}_x}$) and
$$
\|\cos(t\sqrt H)P_c f\|_{L^r_t L^q_x} \les \|H^{s/2} f\|_{L^2_x} \les \|f\|_{H^s_x}.
$$
\end{proposition}
\begin{proof}
Indeed, the side $(\frac 1 q, 0)$, $2 \leq q < \infty$ corresponds to Sobolev embeddings, while the side $(\frac 1 q, \frac 1 q)$, $8 \leq q < \infty$, corresponds to inequalities proved in Theorem \ref{dispersive}. Estimates inside the triangle follow by interpolation. 
\end{proof}

Next, we state the pointwise decay estimates. Such estimates can presumably be used to prove further decay estimates, such as Strichartz estimates.

We first define the Kato-type spaces of potentials
$$
\|V\|_{\tilde {\mc K}_{1/2}} := \sup_y \int_{|x-y| \leq 1} \frac {|V(x)|\dd x}{|x-y|^{1/2}},\ \|V\|_{\tilde {\mc K}} := \sup_y \int_{|x-y| \leq 1} \frac {|V(x)| \log|x-y| \dd x}{|x-y|^{1/2}}.
$$
\begin{theorem}\lb{decay_estimate} Assume that $(1+\log_+|x|)^3 V \in L^1_x$, for some $q \in (1, \infty]$ $V \in L^q_{loc}$ and $\lim_{R \to \infty} R^{1-1/q} \|V\|_{L^q(|x|\in [R, 2R])}=0$, $0$ is a regular point of the spectrum of $H=-\Delta+V$, and in addition that $(1+\log_+|x|) V \in \tilde {\mc K}_{1/2}$, $V \in \tilde {\mc K}$. Then
\be\lb{H1}
\|\cos(t\sqrt H) P_c \langle H \rangle^{-3/4} f\|_{L^\infty_x} + \|\frac{\sin(t\sqrt H) P_c}{\sqrt H} \langle H \rangle^{-1/4} f\|_{L^\infty_x} \les \|f\|_{\mc H^1_x},
\ee
where $\mc H^1_x$ is the Hardy space.

Furthermore,
\be\lb{L^1}\begin{aligned}
\|\cos(t\sqrt H) P_c f\|_{L^\infty_x} &\les \|\langle H \rangle^{3/4+} f\|_{L^1_x} \les \|f\|_{W^{3/2+, 1}_x},\\
\left\|\frac{\sin(t\sqrt H) P_c}{\sqrt H} f\right\|_{L^\infty_x} &\les \|\langle H \rangle^{1/4+}f\|_{L^1_x} \les \|f\|_{W^{1/2+, 1}_x}.
\end{aligned}\ee
\end{theorem}
The Kato-type conditions on $V$ are satisfied when $(1+\log_+|x|) V \in L^{4/3}_x$ and $V \in L^{4/3+}_x$.

Under slightly more restrictive conditions on the potential $V$ we have the following refinement of several results of Theorem \ref{dispersive}, which are also almost scaling-invariant.

\begin{proposition}\lb{other_estimates} Assume that $(1+\log_+|x|)^4 V \in L^1_x$, for some $q \in (1, \infty]$ $V \in L^q_{loc}$ and $\lim_{R \to \infty} R^{1-1/q} \|V\|_{L^q(|x|\in [R, 2R])}=0$, and $0$ is a regular point of the spectrum of $H=-\Delta+V$. Then for $t \geq 1$
\be\lb{d1}
\int_t^\infty \left|\frac {\sin(\tau\sqrt H) P_c}{\sqrt H} \right|(x, y) \dd \tau \les (1+\log_+t)^{-1} (1+\log_+|x|)^2 (1+\log_+|y|)^2,
\ee
\be\lb{d2}
(1+\log_+|x|)^{-2}|\cos(t\sqrt H) P_c f|(x) \les (1+\log_+t)^{-1} \|(1+\log_+|x|)^2(-\Delta f)\|_{L^1_x}.
\ee
Also, for $t\geq 0$
\be\lb{d3}
(1+\log_+|x|)^{-2} \int_t^\infty |\cos(\tau\sqrt H) P_c f|(x) \dd \tau \les (1+\log_+t)^{-1} \|(1+\log_+|x|)\nabla f\|_{L^1_x},
\ee
\be\lb{d4}
(1+\log_+|x|)^{-2} \left|\frac{\sin(t\sqrt H)P_c}{\sqrt H} f\right|(x) \les (1+\log_+t)^{-1} \|(1+\log_+|x|)\nabla f\|_{L^1_x}.
\ee
\end{proposition}

Note that other estimates with different combinations of norms are obtainable. The proof of (\ref{c11}) and (\ref{c12}) is based on the following lemma:
\begin{lemma}\lb{lema}
Under the same conditions on $V$, for $1/4<s<1$ we have
\be\lb{c10}
\left\|\int_{-\infty}^t \frac {e^{i(t-\tau)\sqrt H} P_c}{H^s} F(\tau) \dd \tau \right\|_{L^{q_1, \sigma}_x L^{r_1, \tilde \sigma}_t} \les \|F\|_{L^{q_2, \sigma}_x L^{r_2, \tilde \sigma}_\tau},
\ee
where $1 \leq q_1, q_2, r_1, r_2 \leq \infty$ and $\frac 2 {q_1} + \frac 1 {r_1} + 2s + 1 = \frac 2 {q_2} + \frac 1 {r_2}$.

For $1/4<s<1/2$ $0 \leq \frac 1 {r_2} - \frac 1 {r_1} \leq \frac {4s-1} 2$, with Lorentz space modifications when $r_1=\infty$ or $r_2=1$ and extra modifications when $\frac 1 {r_2} - \frac 1 {r_1} = \frac {4s-1} 2$ and $r_1=\infty$ or $r_2=1$. For $1/2 \leq s \leq 3/4$ $2s-1 < \frac 1 {r_2} - \frac 1 {r_1} \leq \frac {4s-1} 2$, except there is no equal sign when $s=3/4$ and there are Lorentz space modifications when $r_1=\infty$ or $r_2=1$ and extra modifications when $\frac 1 {r_2} - \frac 1 {r_1} = \frac {4s-1} 2$ and $r_1=\infty$ or $r_2=1$. For $3/4<s<1$ $2s-1<\frac 1 {r_2} - \frac 1 {r_1} \leq 1$.

When $q_1=\infty$ $L^{q_1, \sigma}=L^\infty$ and $L^{q_2, \sigma}=L^{q_2, 1}$ and when $q_2=1$ $L^{q_1, \sigma}=L^{q_1, \infty}$ and $L^{q_2, \sigma}=L^1$.
\end{lemma}

Finally, we provide a semilinear application of the reversed Strichartz estimates (\ref{c6}) and (\ref{c11}-\ref{c14}).
\begin{proposition}\lb{prop_semilinear} Let $p_\ast = 7$ and assume $p_\ast\leq p<\infty$. Consider the semilinear equation
\be\lb{semiliniar}
f_{tt}-\Delta f + V f = G_p(f) + F,\ f(0)=f_0,\ f_t(0)=f_1,
\ee
where
$$
|G_p(f)| + |f| |G_p'(f)| \les |f|^p
$$
and $(1+\log_+|x|)^2V \in L^1_x$, for some $q>1$ $V\in L^q_x$, and $\lim_{R \to \infty} R^{1-1/q}\|V\|_{L^q(|x|\in [R, 2R])} = 0$. Assume that $H=-\Delta+V$ has no eigenvalue or resonance (in particular zero is a regular point of the spectrum).

Let $s=\frac{p-3}{p-1}$. If $f_0 \in H^s$, $f_1 \in \dot H^{s-1+} \cap \dot H^{s-1-}$, and $\displaystyle F \in \bigcup_{\substack{\frac 2 {q_2} + \frac 1 {r_2} = 3-s \\ 1 < q_2 \leq \infty, \frac 4 {3-s} \leq r_2 \leq \infty}} L^{q_2, 2}_x L^{r_2, 2}_t$ are sufficiently small in norm, then equation (\ref{semiliniar}) has global solutions
$$
\displaystyle f \in \bigcap_{\substack {\frac 2 {q_1} + \frac 1 {r_1} = 1-s \\ 8 \leq q_1 < \infty, \frac 4 {1-s} \leq r_1 \leq \infty}} L^{q_1, 2}_x L^{r_1, 2}_t,
$$
where $L^{r_1, 2}=L^\infty$ when $r_1=\infty$ and the endpoint $(q_1, r_1) = (8, \infty)$ is excluded.

If in addition $1 < q_2 \leq \frac 8 7$ and $r_2 \leq 2$ (where $L^{r_2, 2}=L^1$ when $r_2=1$, also excluding the endpoint $(q_2, r_2)=(\frac 8 7, 1)$, then $f \in L^\infty_x (\dot H^{s+}_x + \dot H^{s-}_x)$ and $f_t \in L^\infty_t H^{s-1}_x$.
\end{proposition}

The critical Sobolev regularity for this theorem is $s_\ast = \frac 2 3$.

\begin{proof} Note that $p=\frac {3-s}{1-s}$ and that when $p=p_1:= \frac {13+\sqrt{41}}{-3+\sqrt{41}} \simeq 5.7$ then $s=s_1:=\frac {11-\sqrt{41}} 8 \simeq .57$ is the exponent for which the equality $\frac 2 8 + \frac {1-s}{3-s} = 1-s$ holds. For this value of $s=s_1$ one has that $f^{p_1} \in L^{\frac 8 {p_1}}_x L^1_t$, which is in this case a dual (reversed) Strichartz norm.

However, for (\ref{c6}) to hold, the exponents must also satisfy the condition $\frac 1 {r_2} - \frac 1 {r_1} \leq \frac 1 2$. Since $r_2 = \frac {1-s}{3-s} r_1$, this means that $r_1 \geq \frac 4 {1-s}$ and $r_2 \geq \frac 4 {3-s}$. Combined with $\frac 2 {q_1} + \frac 1 {r_1} = 1-s$ and $q_1 \geq 8$, this implies that $s \geq s_\ast = \frac 2 3$. The corresponding power of the nonlinearity is $p_\ast = 7$.

The conclusion follows by a standard fixed point argument.
\end{proof}

\subsection{History of the problem}
The study of Strichartz estimates for the wave equation began with the paper \cite{strichartz}. Dispersive estimates for the wave equation were proved by, among others, \cite{pecher}, \cite{beals2}, \cite{beals}, \cite{kapitanski}, \cite{ginibre}, \cite{keel}, \cite{cuccagna}, \cite{georgiev}, \cite{ancona}, \cite{cardoso2}, \cite{kopylova}, and \cite{green}. For the high frequency portion of the solution only, estimates were also proved by \cite{cardoso} and \cite{moulin}.

By comparison to higher dimensions, the two-dimensional case, which is the lowest dimension in which dispersive estimates hold for the wave equation, presents a series of special features that contribute to making it a special case.

In two dimensions, the solution to the free wave equation (when $V=0$) is given explicitly by the formula (for $t \geq 0$)
$$
\left(\frac{\sin(t\sqrt{-\Delta})}{\sqrt{-\Delta}} f\right)(x) = \frac 1 {2\pi} \int_{|x-y| \leq t} \frac {f(y) \dd y}{\sqrt{t^2-|x-y|^2}}.
$$
This propagator has a decay rate of $t^{-1}$, which barely fails to be integrable. As we shall see in Lemma \ref{free_resolvent}, this is entirely due to the existence of a resonance at zero energy for the free Laplacian in dimension two, namely the constant function $1$. This is also related to the fact that the resolvent $R_0(\lambda)$ is not uniformly bounded as $\lambda \to 0$.

Another feature that complicates the analysis in two dimensions is the failure of the endpoint Sobolev embedding $\dot H^1 \subset L^\infty$ for the energy space.

It is well-known (see e.g.\ \cite{keel}) that the solution $f$ to the free wave equation
$$
f_{tt}-\Delta f = F,\ f(0)=f_0,\ f_t(0)=f_1
$$
satisfies the following Strichartz-type estimates in dimension two:
$$
\|f\|_{C_t \dot H^s \cap L^{q_1}_t L^{r_1}_x} + \|f_t\|_{C_t \dot H^{s-1}} \les \|f_0\|_{\dot H^s} + \|f_1\|_{\dot H^{s-1}} + \|F\|_{L^{q_2'}_t L^{r_2'}_x},
$$
under the assumptions $2 \leq q_1, q_2 \leq \infty$, $2 \leq r_1, r_2 < \infty$, $\frac 2 {q_1}  + \frac 1 {r_1} \leq \frac 1 2$ and same for $(q_2, r_2)$, and
$$
\frac 1 {q_1} + \frac 2 {r_1} = 1-s = \frac 1 {q_2'} + \frac 2 {r_2'} - 2.
$$ 
Here $q_2'$ and $r_2'$ denote the conjugate exponents to $q_2$ and $r_2$. These estimates also carry over to the perturbed case (when there is a potential) via the $L^p$, $1<p<\infty$, boundedness of wave operators, see \cite{yajima}.

Giving up some regularity in the angular variable $\theta$, one obtains instead the following modified estimates (see \cite{smith}) that hold more generally when $q_1, q_2, r_1, r_2 > 2$, $(q_1, r_1)$, $(q_2, r_2) \ne (\infty, \infty)$, $\frac 1 {q_1} + \frac 1 {r_1}  < \frac 1 2$ and same for $(q_2, r_2)$:
$$
\|f\|_{C_t \dot H^s \cap L^{q_1}_t L^{r_1}_{|x|} L^2_\theta} \les \|f_0\|_{\dot H^s} + \|f_1\|_{\dot H^{s-1}} + \|F\|_{L^{q_2'}_t L^{r_2'}_{|x|} L^2_\theta}.
$$

In addition, the solution to the free wave equation satisfies the estimate (see \cite{beals})
$$
\|\cos(t\sqrt{-\Delta}) f\|_{L^\infty_x} \les |t|^{-1/2} \|f\|_{W^{3/2+, 1}_x},\ \left\|\frac{\sin(t\sqrt{-\Delta})}{\sqrt{-\Delta}} f\right\|_{L^\infty_x} \les |t|^{-1/2} \|f\|_{W^{1/2+, 1}_x},
$$
with the endpoint estimate becoming true if we replace $L^1$ by the Hardy space $\mc H^1$.

Recently, \cite{green} proved pointwise decay estimates for the wave equation with potential in dimension two such as
$$
\|\cos(t\sqrt H) P_c \langle H \rangle^{-3/4-} f\| + \left\|\frac {\sin(t\sqrt H) P_c}{\sqrt H}\langle H \rangle^{-1/4-} f\right\|_{L^\infty_x} \les \|f\|_{L^1_x},
$$
under the assumption that $|V(x)| \les \langle x \rangle^{-3-}$. \cite{green} also proved weighted, integrable decay estimates such as
$$\begin{aligned}
\|\cos(t\sqrt H) P_c \langle H \rangle^{-3/4-} f\|_{\langle x \rangle^{1/2+} L^\infty_x} + \left\|\frac {\sin(t\sqrt H) P_c}{\sqrt H}\langle H \rangle^{-1/4-} f\right\|_{\langle x \rangle^{1/2+} L^\infty_x} \les \\
\les t^{-1} (\log t)^{-2} \|f\|_{\langle x \rangle^{-1/2-} L^1_x}
\end{aligned}$$
for $t>2$. In addition, \cite{green} proved decay estimates for the case when there are zero energy eigenvalues or resonances.

\cite{moulin} obtained by entirely different methods the high frequency estimate
$$
\|e^{it\sqrt H} H^{-3/4-} (1-h(H/R)) f\|_{L^\infty_x} \les t^{-1/2} \|f\|_{L^1_x}
$$
for sufficiently large $R$, under the sole condition that
$$
\|V\|_{\tilde {\mc K}_{1/2}} := \sup_y \int \frac{|V(x)| \dd x}{|x-y|^{1/2}} < \infty.
$$
Here $h$ is a smooth cutoff function.

In the current paper we improve upon the results of \cite{green} in several ways, by reducing the necessary regularity and decay conditions on the potential and by reducing the weights needed for the integrable decay estimates, both in space and in frequency (see Theorem \ref{dispersive} and Proposition \ref{other_estimates}). In addition, for the $t^{-1/2}$ decay estimate we obtain the endpoint case, which involves $\mc H^1$ (see Theorem \ref{decay_estimate}). We also prove reversed and ordinary Strichartz estimates.


\section{Preliminaries}
\subsection{Notations} We denote by $A \les B$ the inequality $|A| \leq C |B|$ for some constant $C$, which may change from line to line, $a \vee b = \max(a, b)$, $a \wedge b = \min(a, b)$, and $\langle x \rangle := \sqrt {1+x^2}$.

We define $\B(X, Y)$ to be the Banach space of bounded linear operators from the Banach space $X$ to $Y$ and $\B(X)=\B(X, X)$.

Let $R_0(\lambda) := (-\Delta-\lambda)^{-1}$, $R_V(\lambda) := (-\Delta+V-\lambda)^{-1}$, and $P_c$ be the projection on the continuous spectrum of $H=-\Delta+V$.

We adopt the following convention for the Fourier transform:
$$
\mc F f (\xi) = \widehat f(\xi) := \int_{-\infty}^\infty f(x) e^{-ix\xi} dx,\ \mc F^{-1} g(x) = \widecheck g(x) := \frac 1 {2\pi} \int_{-\infty}^\infty g(\xi) e^{ix\xi} d\xi.
$$
Let $F(\nabla) f := \mc F^{-1}(F(\xi) \widehat f(\xi))$.

In addition, denote the Lebesgue spaces by $L^p$, the Lorentz spaces by $L^{p, q}$ (see \cite{bergh} for their definition and properties; note that $L^{p, q_1} \subset L^{p, q_2}$ for $q_1 \leq q_2$, $L^{p, p}=L^p$, and $L^{p, \infty}$ is weak-$L^p$), the Hardy space by $\mc H^1$, the Sobolev spaces
$$
\dot H^s = \{f \mid |\xi|^s \widehat f (\xi) \in L^2_\xi\},\ H^s = \{f \mid \langle \xi \rangle^s f(\xi) \in L^2_\xi\},\ W^{s, p} = \{f \mid \langle \nabla \rangle^s f \in L^p\}
$$
and weighted spaces $F(x) L^p_x = \{F(x) f(x)\mid f(x) \in L^p_x\}$.

Define the local Kato space
\be\lb{local_kato}
\mc K_\theta = \{f \in \mc M \mid \|f\|_{\mc K_\theta} := \sup_y \int_{\R^2} |f(x)| \log_-^\theta|x-y| \dd x < \infty
\ee
where $\log_-:=-\min(\log, 0)$ and $\log_+:=\max(\log, 0)$. This is a Banach space of measures. Note that $L^p \subset \mc K_\theta$ for $1<p \leq \infty$.

The dual of $\mc K_\theta$ is
\be\lb{kato*}
\mc K^*_\theta = \{f \mid f(x) = \int_{\R^2} g(x, y) \one_{|x-y| \leq 1} d\mu(y) \mid \int_{\R^2} \left(\sup_{|x-y| \leq 1} \frac {|g(x, y)|}{\log_-^\theta|x-y|}\right) d\mu(y) < \infty\}.
\ee
Then $\mc K^*_\theta \subset L^p$ for $1\leq p<\infty$.

The reason for the definitions of $\mc K_\theta$ and $\mc K^*_\theta$ is that the operator with kernel $\log^\theta|x-y|$ is bounded from $(1+\log_+|x|)^{-\theta} L^1_x \cap \mc K_\theta$ to $(1+\log_+|x|)^\theta L^\infty_x + \mc K^*_\theta$ (where one of $\mc K_\theta$ and $\mc K_\theta^*$ can be missing).

At a minimum, assuming that $V \in \mc K_1$, then $|V(x)| \log_-|x-y| \in \B(L^1)$, $\log_-|x-y| |V(y)|\in \B(L^\infty)$, and by interpolation $v(x) \log_-|x-y| v(y) \in \B(L^2)$, where $v=|V|^{1/2}$. We also need that $v \in \B(L^2, \langle 1+\log_+|x| \rangle^{-1} L^1_x)$. This means that we require that $V \in \langle 1+\log_+|x| \rangle^{-2} L^1_x$. In addition, in order to be able to assume the absence of embedded eigenvalues, we require that for some $q \in (1, \infty]$ $V \in L^q_{loc}$ and
$$
\lim_{R \to \infty} R^{1-1/q} \|V\|_{L^q(|x| \in [R, 2R)} = 0.
$$
This is strictly stronger than the $V \in \mc K_1$ condition we would need otherwise.

We also define the following two local Kato-type norms:
$$
\|f\|_{\tilde {\mc K}_{1/2}} := \sup_y \int \frac {|V(x)| \dd x}{|x-y|^{1/2}},\ \|f\|_{\tilde {\mc K}} := \sup_y \int \frac {|V(x)| \log_-|x-y|}{|x-y|^{1/2}}.
$$

For a Banach lattice $X$ (a Banach space of functions or measures such that if $|f| \leq |g|$ then $\|f\|_X \leq \|g\|_X$), consider the Banach spaces of kernels
$$
\U_{X, Y} := \{T(t, y, x): M(T)(y, x) := \int_{-\infty}^\infty |T(t, y, x)| \dd t \in \B(X_x, Y_y)\}
$$
and in particular $\U_X:=\U_{X, X}$. A kernel in $T \in \U_{X, Y}$ admits an inverse Fourier transform $\widecheck T(\lambda) \in C_\lambda \B(X, Y)$. We denote the Banach space of such Fourier transforms by $\widehat {\U_{X, Y}}$.

The spaces $\U_{X, Y}$ form an algebroid under convolution, in the sense that for any three Banach lattices $X$, $Y$, and $Z$
$$
\|T_1 \ast T_2\|_{\U_{X, Z}} := \left\|\int_{-\infty}^\infty T_1(t-s, z, y) T_2(s, y, x) \dd s\right\|_{\U_{X, Z}} \leq \|T_1\|_{\U_{Y, Z}} \|T_2\|_{\U_{X, Y}}.
$$

In addition to $\U_{X, Y}$, we consider the spaces $(1+\log_+|t|)^{-k} \U_{X, Y}$, which have improved smoothness in the dual variable $\lambda$ and also form an algebroid for each $k \in \R$.

\subsection{Basic properties} The improved smoothness announced above is captured by the following inequality:
\begin{lemma}
For a smooth cutoff function $h(\lambda)$
\be\lb{smoothness}\begin{aligned}
&\|(\widehat f(\lambda) - \widehat f(0)) h(\lambda) \log |\lambda|\|_{\widehat {L^1}_\lambda} + \\
&+ \|(\widehat f(\lambda) - \widehat f(0)) h(\lambda) \sgn \lambda\|_{\widehat {L^1}_\lambda} \les \|f\|_{(1+\log_+|t|)^{-1} L^1_t}.
\end{aligned}\ee
\end{lemma}
\begin{proof}
Indeed, let $\eta(t)$ be a smooth cutoff function such that $0 \leq \eta \leq 1$ and $\eta(t)=0$ for $t<1-\epsilon$, $\eta(t)=1$ for $t>1+\epsilon$. Then
$$
((\eta(t) - \eta(-t))t^{-1})^\vee(\lambda) = - \frac 1 \pi \log |\lambda| +g(\lambda) = -\frac 1 \pi h(\lambda) \log |\lambda| +\tilde g(\lambda)
$$
and it decays to any order away from $0$, where $g(\lambda)$ and $\tilde g(\lambda)$ are smooth functions and $h(\lambda)$ is a smooth cutoff function. Note that
$$
\widehat f(\lambda)-\widehat f(0)=\int_{-\infty}^\infty (e^{-it_0\lambda}-1) f(t_0) \dd t_0
$$
and
$$
((e^{it_0\lambda}-1) (-\frac 1 \pi h(\lambda) \log |\lambda| + \tilde g(\lambda)))^{\wedge}(t) = (\eta(t-t_0)-\eta(t_0-t)) (t-t_0)^{-1} - (\eta(t)-\eta(-t)) t^{-1}
$$
has $L^1_t$ norm of at most $C_1+C_2\log_+|t_0|$. The term $\tilde g(\lambda)$ has a contribution of bounded $L^1_t$ norm. Therefore the first half of (\ref{smoothness}) follows. The second half is obtained in a similar manner, taking into account that
$$
((\eta(t) + \eta(-t))t^{-1})^\vee(\lambda)=\frac i 2 \sgn \lambda + \tilde g(\lambda).
$$
\end{proof}

Let $\U_2 := \U_{L^2}$.

Below we establish some fundamental properties of the resolvent in dimension two that we use in the proof.
\begin{lemma}\lb{free_resolvent} The resolvent $R_0((\lambda+i0)^2)$ admits a decomposition into
$$
R_0((\lambda+i0)^2)(x, y) = \widecheck L(\lambda)(x, y) + h(\lambda) (\frac i 4 \sgn \lambda - \frac 1 {2\pi} \log |\lambda|),
$$
where $\one_{[0, 1]}(t) L(t, x, y) \in \U_{\mc K_1, L^\infty_x} \cap \U_{L^1_x, \mc K_1^*}$, $\one_{[1, \infty)}(t) (1+\log_+|t|)^k L(t, x, y) \in \\$
$\U_{(1+\log_+|x|)^{-k-1} L^1_x, (1+\log_+|x|)^{k+1} L^\infty_x}$, and $h(\lambda)$ is a smooth cutoff function such that $h(0)=1$. Furthermore, for $v \in (1+\log_+|x|)^{-1} L^2_x$,
\be\lb{cond_Wiener}\begin{aligned}
&(i)\ \lim_{\delta \to 0} \|L(t+\delta)-L(t)\|_{\U_{(1+\log_+|x|)^{-1} L^1_x \cap \mc K_{1+}, (1+\log_+|x|) L^\infty_x}} = 0,\\
&(ii)\ \lim_{R \to \infty} \|\one_{|t| \geq R} v L(t) v\|_{\U_2} = 0,
\end{aligned}\ee
and for small $\epsilon$ one can choose $L$ such that
$$
\widecheck L(0)(x, y) = - \frac 1 {2\pi} \log |x-y| + C + O(\epsilon).
$$
\end{lemma}

\begin{proof}
Recall the following formula expressing the free resolvent in terms of Hankel functions:
$$
R_0^\pm(\lambda^2)(x, y) = \pm \frac i 4 H_0^\pm(\lambda|x-y|).
$$
Hankel functions are the inverse Fourier transforms of the following functions (see the Appendix):
\be\lb{hankel1}
H_0^+(\rho) = \frac 2 {i\pi} \int_1^\infty (t^2-1)^{-1/2} e^{i\rho t} \dd t,
\ee
respectively
\be\lb{hankel2}
H_0^-(\rho) = -\frac 2 {i\pi} \int_{-\infty}^{-1} (t^2-1)^{-1/2} e^{i\rho t} \dd t.
\ee
Consequently we can express the free resolvent as an inverse Fourier transform of a family of integral kernels by the following formula:
$$
R_0^+(\lambda^2)(x, y) = \frac 1 {2\pi} \int_{|x-y|}^\infty (t^2-|x-y|^2)^{-1/2} e^{i\lambda t} \dd t := \frac 1 {2\pi} \int_{-\infty}^\infty K(t, x, y) e^{i\lambda t} \dd t
$$
and similarly for $R_0^-(\lambda^2)$:
$$
R_0^-(\lambda^2)(x, y) = \frac 1 {2\pi} \int_{-\infty}^{-|x-y|} (t^2-|x-y|^2)^{-1/2} e^{i\lambda t} \dd t = \frac 1 {2\pi} \int_{-\infty}^\infty K(t, x, y) e^{-i\lambda t} \dd t.
$$
We combine both into a single formula for $R_0((\lambda+i0)^2)$.

The family of kernels $\frac 1 {2\pi} K(t, x, y)=\frac 1 {2\pi} \one_{|x-y| \leq t} (t^2-|x-y|^2)^{-1/2}$ is the Green's function for the wave equation in two dimensions. Note that for every $x$ and $y$ $K(t, x, y)$ is almost but not quite integrable in $t$, having a tail of size $t^{-1}$ at $+\infty$. In fact, $K(t, x, y)$ belongs to $L^{1, \infty}_t \cap L^{2, \infty}_t$ and to $L^p_t$ for $1<p<2$.

We decompose $K(t, x, y)$ into an integrable component and a a rank-one nonintegrable tail:
$$
K(t, x, y) = (\one_{t \geq |x-y|} (t^2-|x-y|^2)^{-1/2} - \eta(t) t^{-1}) + \eta(t) t^{-1},
$$
where $\eta$ is a smooth function such that $0 \leq \eta(t) \leq 1$, $\eta(t) = 0$ for $t \leq 1-\epsilon$, and $\eta(t) = 1$ for $t \geq 1+\epsilon$, for some small $\epsilon$ to be chosen later.

Note that
$$
\one_{t \geq |x-y|} (t^2-|x-y|^2)^{-1/2} - \eta(t) t^{-1} = \one_{t \geq |x-y|} ((t^2-|x-y|^2)^{-1/2} - t^{-1}) + \one_{t \geq |x-y|} t^{-1} - \eta(t) t^{-1},
$$
where
$$
|\int_0^\infty \one_{t \geq |x-y|} t^{-1} - \eta(t) t^{-1} \dd t + \ln |x-y|| \les \epsilon
$$
and independently of $x$ and $y$, by rescaling,
$$
\int_0^\infty \one_{t \geq |x-y|} ((t^2-|x-y|^2)^{-1/2} - t^{-1}) \dd t = C \leq \int_0^\infty \one_{t \geq |x-y|} \frac {|x-y|^2}{t^2\sqrt{t^2-|x-y|^2}} \dd t = \tilde C.
$$
Therefore
$$
\int_0^\infty |\one_{t \geq |x-y|} (t^2-|x-y|^2)^{-1/2} - \eta(t) t^{-1}| \dd t \leq C + |\log|x-y||.
$$
Furthermore, taking into account the fact that for $e^k\leq|x-y|\leq t$ $\frac {\ln^k t}{\ln^k |x-y|} \leq \frac t {|x-y|}$, for $|x-y| \geq e^k$ we obtain
$$
\int_{|x-y|}^\infty (\log_+ t)^k \frac {|x-y|^2}{t^2\sqrt{t^2-|x-y|^2}} \dd t \les \log_+^k|x-y|.
$$
Together with
$$
|\int_0^\infty (\log_+ t)^k (\one_{t \geq |x-y|} t^{-1} - \eta(t) t^{-1}) \dd t| \les 1 + \log_+^{k+1} |x-y|,
$$
it implies that
$$
\int_0^\infty (\log_+ t)^k |\one_{t \geq |x-y|} (t^2-|x-y|^2)^{-1/2} - \eta(t/\epsilon) t^{-1}| \dd t \les 1+ \log_+^{k+1}|x-y|.
$$
Note that the inverse Fourier transform of $\eta(t) t^{-1}$ is of the form
$$
(\eta(t) t^{-1})^\vee(\lambda) = \frac i 4 \sgn \lambda - \frac 1 {2\pi} \log |\lambda| +g(\lambda) = h(\lambda)(\frac i 4 \sgn \lambda - \frac 1 {2\pi} \log |\lambda|)+\tilde g(\lambda)
$$
and it has decay to any order away from $0$, where $g(\lambda)$ and $\tilde g(\lambda)$ are smooth functions and $h(\lambda)$ is a smooth cutoff function. We then set
$$
L(t)(x, y) = (\one_{t \geq |x-y|} (t^2-|x-y|^2)^{-1/2}-\eta(t)t^{-1})+\widehat {\tilde g}(t).
$$

Finally, we prove (\ref{cond_Wiener}) as follows: for (i), note that
$$
\int_0^\infty |L(t+\delta)-L(t)|(x, y) \dd t \les \min(C+|\log|x-y||, \frac {\delta}{|x-y|}+\frac {\delta^{1/2}}{|x-y|^{1/2}}) + o(1),
$$
while for (ii) we see that
$$
\int_R^\infty |L(t)|(x, y) \dd t \les \chi_{|x-y| \geq R} (\log|x-y|-\log R) + \max(1, \frac {|x-y|^2}{R^2}) + o(1).
$$
\end{proof}

\section{Proof of the main statement}
Our proof uses the symmetric resolvent identity
\be\lb{sym_id}
R_V((\lambda+i0)^2) = R_0((\lambda+i0)^2) - R_0((\lambda+i0)^2) v (U+T(\lambda))^{-1} v R_0((\lambda+i0)^2),
\ee
where $v=|V|^{1/2}$, $U=\sgn V$ (defined so that $U^2=1$), and $T(\lambda):= v R_0((\lambda+i0)^2) v$. Other useful identities include
$$
R_V((\lambda+i0)^2) = (I + R_0((\lambda+i0)^2) V)^{-1} R_0((\lambda+i0)^2)
$$
and
\be\lb{assym_id}
(I + R_0((\lambda+i0)^2) V)^{-1} = I - R_0((\lambda+i0)^2) v (U+T(\lambda))^{-1} v.
\ee

We compute the Fourier transform of $(U+T(\lambda))^{-1}=(U + v R_0((\lambda+i0)^2) v)^{-1}$ as a function of $\lambda$ and show that it is in $\U_2$, except for some specific terms that appear in the expansion near $\lambda=0$.

First, we need to show that $(U+T(\lambda))^{-1}$ exists pointwise for each $\lambda \ne 0$.

\begin{lemma}\lb{invertibility} Assume that $V \in L^1_x$ is real-valued and such that for some $q \in (1, \infty]$ $V \in L^q_{loc}$ and
$$
\lim_{R \to \infty} R^{1-1/q} \|V\|_{L^q(|x| \in [R, 2R])} = 0.
$$
Then for each $\lambda \in \R \setminus \{0\}$ $(U+T(\lambda))^{-1} - U \in \B(L^2)$ is a Hilbert-Schmidt operator.
\end{lemma}
\begin{proof} First, note that $(U+T(\lambda))^{-1} = U^{-1} (I + T(\lambda) U)^{-1}$. Consider the operator $T(\lambda) U = v R_0((\lambda+i0)^2) v U$. It has integral kernel $\pm \frac i 4 v(x) H_0^\pm(\lambda|x-y|) v(y) U(y)$, where the Hankel functions have the asymptotic behavior $H^0_\pm(\rho) = O(\log \rho)$ for $\rho \to 0$ and $H^0_\pm(\rho) = O(\rho^{-1/2})$ for $\rho \to \infty$. Due to our assumptions, $T(\lambda) U = v R_0((\lambda+i0)^2) v U$ is a Hilbert-Schmidt operator, hence compact.

By Fredholm's alternative the only alternative to the invertibility of $I+T(\lambda) U$ would be the existence of a nonzero function $f \in L^2$ such that $f + v R_0((\lambda+i0)^2) v U f = 0$. To preclude this we follow the classical Agmon bootstrap argument.

Let $g:=R_0((\lambda+i0)^2) vUf \in L^\infty$. Then $g + R_0((\lambda+i0)^2) V g = 0$. Since $g \in L^\infty_x$ and $V \in L^1_x$ is real-valued, the pairing
$$
(g, Vg) = - (R_0((\lambda+i0)^2) V g, V g)
$$
takes a real value. However, one has that
$$
\im (R_0((\lambda+i0)^2) V g, V g) = c \lambda \int_{S^1} |\widehat {Vg}(\lambda \omega)|^2 d\omega,
$$
so $\widehat {Vg}(\lambda \omega) = 0$ on the circle. Adapting Proposition 2.4 in \cite{gosc} to the two-dimensional case, one can show that for $1 \leq p < 6/5$ and for any $\delta<\frac 1 2 - \frac 3 {p'}$, if $\widehat F = 0$ on $S^1$ then
$$
\|(1+|x|)^{\delta-1/2}R_0^\pm(1) F\|_{L^2_x} \les \|F\|_{L^p_x}.
$$
Setting $F=Vg \in L^1$, it follows that $(1+|x|)^{\delta-1/2} g \in L^2$ for some $\delta>0$. Note that $g$ satisfies the distributional equation $(-\Delta+V-\lambda^2)g=0$. One then easily checks that $g \in W^{2, q}_{loc} \subset H^1_{loc}$. Consequently we can apply the results of \cite{ioje} or \cite{kota} concerning the absence of embedded eigenvalues, which imply that such a solution $g$ must be zero. Therefore $f=0$ and $(I + T(\lambda) U)^{-1}$ is $L^2$-bounded.

Since we can write
$$
(U+T(\lambda))^{-1} = U - U T(\lambda) (U+T(\lambda))^{-1}
$$
it follows that the difference is a Hilbert-Schmidt operator, as claimed.
\end{proof}

Denote by $G_0(x, y)=-\frac 1 {2\pi} \log |x-y|$ Green's function in two dimensions.

Let $P=\frac 1 {\|V\|_{L^1_x}} \langle \cdot, v \rangle v$ be the orthogonal projection on $v$ in $L^2_x$ and $Q=I-P$. Following \cite{jensen}, \cite{schlag}, \cite{ergr}, \cite{ergr2}, and \cite{green} we give the following definition:
\begin{definition}\lb{regular_point}
We say that zero is a regular point of the spectrum for $H=-\Delta+V$ if $Q(U+vG_0v)Q$ is invertible on $QL^2_x$.
\end{definition}

This corresponds to $R_V(\lambda)$ being uniformly bounded in norm as $\lambda \to 0$.

Now we can state our main technical result, concerning the Fourier transform of the resolvent $R_V$.

\begin{lemma}\lb{technical_lemma} Assume that $(1+\log_+|x|)^2 V \in L^1_x$, for some $q \in (1, \infty]$ $V \in L^q_{loc}$ and $\lim_{R \to \infty} R^{1-1/q} \|V\|_{L^q(|x|\in [R, 2R])}=0$, and $0$ is a regular point of the spectrum of $-\Delta+V$. Then
\be\lb{rv}\begin{aligned}
&\supp (R_V((\lambda+i0)^2))^\wedge \subset [0, \infty), \one_{[0, 1]}(t) (R_V((\lambda+i0)^2))^\wedge (t) \in \U_{L^1_x \cap \mc K_1, L^\infty_x},\\
&\one_{[1, \infty)}(t) (R_V((\lambda+i0)^2))^\wedge (t) \in \U_{(1+\log_+|x|)^{-1} L^1_x, (1+\log|x|) L^\infty_x}.
\end{aligned}\ee
Moreover,
\be\lb{r0v}\begin{aligned}
&\supp ((I + R_0((\lambda+i0)^2) V)^{-1})^\wedge \subset [0, \infty),\ \one_{[0, 1]}(t) ((I + R_0((\lambda+i0)^2) V)^{-1})^\wedge(t) \in \U_{L^\infty_x},\\
&((I + R_0((\lambda+i0)^2) V)^{-1})^\wedge(t) \in \U_{(1+\log_+|x|) L^\infty_x}.
\end{aligned}\ee
In fact, for $(I+R_0((\lambda+i0)^2)V)^{-1}$ one has the stronger estimate
\be\lb{r0v'}\begin{aligned}
(I+R_0((\lambda+i0)^2) V)^{-1} &= (1-h(\lambda)) \widecheck N_0(\lambda) + h(\lambda) \widecheck N_1(\lambda) (I-\frac{1 \otimes |V|}{\|V\|_{L^1_x}}) + \\
&+ \frac {h(\lambda)}{\frac i 4 \sgn \lambda - \frac 1 {2\pi} \log|\lambda|} \widecheck N_2(\lambda) (1 \otimes |V|),
\end{aligned}\ee
where $N_0, N_1, N_2 \in \U_{(1+\log_+|x|)L^\infty_x}$ and $h$ is a cutoff function.
\end{lemma}
\begin{proof}
Take an even cutoff function $h$ such that $0 \leq h \leq 1$, $h=0$ outside $(-2, 2)$, and $h=1$ on $[-1, 1]$.

Note that $v L v \in \U_2$, where $v=|V|^{1/2}$ and $L$ is given by Lemma \ref{free_resolvent}.

We consider separately the cases of high, medium, and low energies.

High energies: We want to show that $(1-h(\lambda/R))(U+T(\lambda))^{-1}$ has a Fourier transform in $\U_2$ for sufficiently large $R$. To begin with, one can express the Fourier transform of $(1-h(\lambda/R))T(\lambda)$ as
$$
\widehat T - R \widehat h(R \cdot) \ast \widehat T = \int_{-\infty}^\infty  R \widehat h(R s) (\widehat T(t) - \widehat T(t-s)) \dd s.
$$
Due to the condition (\ref{cond_Wiener}i), as well as due to the fact $(1-h(\lambda/R)) h(\lambda)(\frac i 4 \sgn \lambda - \frac 1 {2\pi} \log \lambda)$ is a Schwartz-class function, we obtain that this converges to zero in $\U_2$ as $R \to \infty$. We then see that
$$\begin{aligned}
(1-h(\lambda/2R))(U+T(\lambda))^{-1} &= (1-h(\lambda/2R))(U+(1-h(\lambda/R))T(\lambda))^{-1} \\
&= (1-h(\lambda/2R)) \sum_{n=0}^\infty U \big((1-h(\lambda/R))T(\lambda) U\big)^n,
\end{aligned}$$
which converges for sufficiently large $R$ in $\widehat {\U_2}$.

Medium energies: Let $S_\epsilon:=h(\frac{\lambda-\lambda_0} \epsilon)(T(\lambda)-T(\lambda_0))$. A straightforward argument based on condition (\ref{cond_Wiener}ii), as well as on the fact that $h(\frac{\lambda-\lambda_0} \epsilon) h(\lambda)(\frac i 4 \sgn \lambda - \frac 1 {2\pi} \log \lambda)$ is a Schwartz-class function, shows that $\lim_{\epsilon \to 0} \|\widehat {S_\epsilon}\|_{\U_2} = 0$. Then
$$\begin{aligned}
h(\frac{\lambda-\lambda_0} \epsilon) (U+T(\lambda))^{-1} &= h(\frac{\lambda-\lambda_0} \epsilon) (U+T(\lambda_0) + h(\frac{\lambda-\lambda_0} {2\epsilon})(T(\lambda)-T(\lambda_0)))^{-1} \\
&= h(\frac{\lambda-\lambda_0} \epsilon) \sum_{n=0}^\infty (-1)^n (U+T(\lambda_0))^{-1} (S_{2\epsilon} (U+T(\lambda_0))^{-1})^n,
\end{aligned}$$
where $(U+T(\lambda_0))^{-1}$ exists and is a Hilbert-Schmidt operator by Lemma \ref{invertibility}. The series converges in $\widehat {\U_2}$ for small enough $\epsilon$.

Zero energy: This case requires separate treatment and in particular not all the terms obtained in the expansion will be in $\widehat {\U_2}$. We make use of the following lemma (the Fehsbach formula, see for example Lemma 2.3 in \cite{jensen}):
\begin{lemma}\lb{fehsbach} Let $X=X_0\oplus X_1$ be a direct sum decomposition of a vector space $X$. Suppose that a linear operator $L \in \B(X)$ is written in the form
$$
L=\begin{pmatrix} L_{00} & L_{01} \\ L_{10} & L_{11} \end{pmatrix}
$$
in this decomposition and that $L_{00}^{-1}$ exists. Let $C=L_{11} - L_{10} L_{00}^{-1} L_{01}$. Then $L^{-1}$ exists if and only if $C^{-1}$ exists and is given by
$$
L^{-1} = \begin{pmatrix}
L_{00}^{-1} + L_{00}^{-1} L_{01} C^{-1} L_{10} L_{00}^{-1} &
- L_{00}^{-1} L_{01} C^{-1} \\
- C^{-1} L_{10} L_{00}^{-1} & C^{-1}
\end{pmatrix}.
$$
\end{lemma}

We assume that zero is a regular point for $H=-\Delta+V$ and write
$$
U+T(\lambda)=\begin{pmatrix} Q (U + T(\lambda)) Q & Q(U+T(\lambda))P \\ P(U+T(\lambda))Q & P(U+T(\lambda))P \end{pmatrix} =: \begin{pmatrix} L_{00}(\lambda) & L_{01}(\lambda) \\ L_{10}(\lambda) & L_{11}(\lambda) \end{pmatrix}.
$$
First, we observe that $L_{00}(0)$ is invertible on $QL^2_x$. Indeed, as $\lambda$ approaches zero $U+T(\lambda)$ has the form, by Lemma \ref{free_resolvent},
$$\begin{aligned}
&U + T(\lambda) = U + v(x)\widecheck L(\lambda)(x, y)v(y) + h(\lambda) (\frac i 4 \sgn \lambda - \frac 1 {2\pi} \log |\lambda|) v \otimes v\\
&=U + v(x)(-\frac 1 {2\pi} \log|x-y|)v(y) + o_{L^2}(1) + (C + h(\lambda)(\frac i 4 \sgn \lambda - \frac 1 {2\pi} \log |\lambda|)) v \otimes v + O_{L^2}(\epsilon),
\end{aligned}$$
so
$$\begin{aligned}
Q (U+T(\lambda)) Q &= Q (U + v(x)\widecheck L(\lambda)(x, y)v(y)) Q\\
&= Q(U + v(x)(-\frac 1 {2\pi} \log|x-y|)v(y))Q + o_{L^2}(1) + O_{L^2}(\epsilon).
\end{aligned}$$
By choosing $\epsilon$ sufficiently small in the decomposition in Lemma \ref{free_resolvent}, we obtain in fact that $L_{00}(\lambda)$ is invertible for all sufficiently small $\lambda$, due to our regularity assumption at $\lambda=0$. Moreover, for sufficiently small $\epsilon$, $h(\frac \lambda \epsilon) L_{00}^{-1}(\lambda) \in \widehat {\U_2}$. Indeed, letting $\tilde S_\epsilon:=
h(\frac \lambda \epsilon)(L_{00}(\lambda)-L_{00}(0))$, we note that $\lim_{\epsilon \to 0} \|\tilde S_{\epsilon}\|_{\widehat {\U_2}} = 0$, so for $\epsilon$ small the power series expansion
$$\begin{aligned}
h(\frac \lambda \epsilon) L_{00}^{-1}(\lambda) &= h(\frac \lambda \epsilon)(L_{00}(0) + h(\frac \lambda {2\epsilon})(L_{00}(\lambda)-L_{00}(0)))^{-1} \\
&= h(\frac \lambda \epsilon) \sum_{n=0}^\infty (-1)^n L_{00}^{-1}(0) (\tilde S_{2\epsilon} L_{00}^{-1}(0))^n
\end{aligned}$$
converges in $\widehat {\U_2}$.

Also note that $\widehat {L_{01}}$ and $\widehat {L_{10}}$ are in $\widehat{\U_2}$ since they contain a $Q$ projection, which eliminates the singular term, while for $L$ given by Lemma \ref{free_resolvent}
$$
L_{11} = P v \widecheck L(\lambda) v P + h(\lambda)(\frac i 4 \sgn \lambda - \frac 1 {2\pi} \log |\lambda|) v \otimes v.
$$
Consequently
$$
C(\lambda) = (L_{11}(\lambda)-L_{10}(\lambda)L_{00}^{-1}(\lambda)L_{01}(\lambda)) = (f(\lambda) + h(\lambda)(\frac i 4 \sgn \lambda - \frac 1 {2\pi} \log |\lambda|)) v \otimes v,
$$
where $f(\lambda) \in \widehat{L^1}_\lambda$ is bounded near $\lambda =0$ and $h(\lambda)$ is a cutoff function. Since the second term dominates the first, for small $\epsilon$ there exists
$$
h(\frac \lambda \epsilon) C(\lambda)^{-1} = \frac {h(\frac\lambda\epsilon)}{f(\lambda)+\frac i 4 \sgn \lambda - \frac 1 {2\pi} \log |\lambda|} \frac 1 {\|V\|_{L^1_x}^2} v\otimes v.
$$
We expand
\be\lb{expansion}
\frac {h(\frac \lambda \epsilon)} {f(\lambda)+\frac i 4 \sgn \lambda - \frac 1 {2\pi} \log |\lambda|} = \sum_{n=0}^\infty \frac {(-1)^n h(\frac \lambda \epsilon) f^n(\lambda)}{(\frac i 4 \sgn \lambda - \frac 1 {2\pi} \log |\lambda|)^{n+1}}.
\ee
Take $\epsilon>0$; then for $n \geq 1$ $\frac {h(\frac \lambda \epsilon)}{(\frac i 4 \sgn \lambda - \frac 1 {2\pi} \log |\lambda|)^n} \in \widehat{L^1}_\lambda$. The first step is showing that $\frac {h(\frac \lambda \epsilon)}{\log^n|\lambda|} \in \widehat{L^1}_\lambda$. Indeed, on one hand
$$
\left(\frac{h(\frac \lambda \epsilon)}{\log^n|\lambda|}\right)^\wedge(t) \les \epsilon|\log\epsilon|^{-n}.
$$
On the other hand, to compute the decay of the Fourier transform, for $1/\epsilon \les t$ we use stationary phase:
$$\begin{aligned}
\left|\int_{-\infty}^\infty \frac {h(\frac \lambda \epsilon)}{\log^n|\lambda|} e^{-it\lambda} \dd \lambda\right| &= \left|\int_0^\infty \frac {h(\frac \lambda \epsilon)}{\log^n\lambda} \cos(t\lambda) \dd\lambda\right| \\
&= \frac 1 t \left|\int_0^\infty \left(\frac {h'(\frac \lambda \epsilon)} {\epsilon\log^n\lambda} - \frac{nh(\frac \lambda\epsilon)}{\lambda\log^{n+1}\lambda}\right) \sin(t\lambda) \dd \lambda\right| \\
&\les \frac 1 {t^2} \left(\frac 1 {\epsilon |\log\epsilon|^n} + \frac n {\epsilon |\log \epsilon|^{n+1}}\right) + \frac n{t \log^{n+1} t}.
\end{aligned}$$
Here we used the fact that $\sin (t\lambda) \leq t \lambda$, hence the symmetry of the integrand; note that $\frac {h(\frac \lambda \epsilon)\sgn\lambda}{\log|\lambda|} \not \in \widehat{L^1}_\lambda$. In conclusion we get that
\be\lb{L1}
\left\|\frac {h(\frac \lambda \epsilon)}{\log^n|\lambda|}\right\|_{\widehat{L^1}_\lambda} \les \frac 1 {|\log \epsilon|^n} + \frac n {|\log \epsilon|^{n+1}}.
\ee

Next, again by the method of stationary phase we obtain that for fixed $c_1$ and $c_2$ and $n \geq 2$
\be\lb{L1'}\begin{aligned}
\left\|\frac {\one_{[0, \infty)}(\lambda)h(\frac \lambda \epsilon)}{(\log|\lambda|+c_1)^{n-k}(\log|\lambda|+c_2|)^k}\right\|_{\widehat{L^1}_\lambda} + \left\|\frac {\one_{(-\infty, 0]}(\lambda)h(\frac \lambda \epsilon)}{(\log|\lambda|+c_1)^{n-k}(\log|\lambda|+c_2|)^k}\right\|_{\widehat{L^1}_\lambda} \les \\
\les \frac 1 {n|\log\epsilon|^{n-1}} + \frac 1 {|\log \epsilon|^n} + \frac n {|\log \epsilon|^{n+1}}.
\end{aligned}\ee
Combining (\ref{L1}) and (\ref{L1'}) leads to
\be\lb{log}
\left\|\frac{h(\frac\lambda \epsilon)}{(\frac i 4 \sgn \lambda - \frac 1 {2\pi} \log |\lambda|)^n}\right\|_{\widehat {L^1}_\lambda} \les \frac 1 {|\log \epsilon|^n} + \frac n {|\log \epsilon|^{n+1}} + \frac {n^2} {|\log \epsilon|^{n+2}}.
\ee
Since
$$
\|h(\frac \lambda {2\epsilon}) f^n(\lambda)\|_{\widehat{L^1}_\lambda} \les C^n \|h(\frac \lambda {4\epsilon}) f(\lambda)\|^n_{\widehat{L^1}_\lambda},
$$
this means that for sufficiently small $\epsilon$ the tail of (\ref{expansion}) converges in $\widehat {L^1}_\lambda$. Therefore
\be\lb{C}\begin{aligned}
h(\frac \lambda \epsilon) C(\lambda)^{-1} &= \frac {h(\frac\lambda \epsilon)}{\frac i 4 \sgn \lambda-\frac 1 {2\pi} \log |\lambda|} \frac 1 {\|V\|_{L^1_x}^2} v \otimes v + \frac {h(\frac\lambda \epsilon) \widecheck k_{11}(\lambda)}{(\frac i 4 \sgn \lambda-\frac 1 {2\pi} \log |\lambda|)^2} P,
\end{aligned}\ee
where $k_{11} P \in \U_2$ (and all powers greater than two have been absorbed into the quadratic term). Denoting in the Fehsbach formula
$$
h(\frac\lambda\epsilon)(U+T(\lambda))^{-1} := \begin{pmatrix} K_{00}(\lambda) & K_{01}(\lambda) \\ K_{10}(\lambda) & K_{11}(\lambda),
\end{pmatrix}
$$
we have obtained an expansion for $K_{11}(\lambda)$. The other matrix entries can then be written as follows: to begin with, by (\ref{C}) and (\ref{log}) $K_{00}(\lambda) = \widecheck k_{00}(\lambda)$, where $k_{00} \in \U_2$. Similarly we obtain that
$$
K_{01}(\lambda) = \frac {h(\frac \lambda \epsilon)}{\frac i 4 \sgn \lambda - \frac 1 {2\pi} \log |\lambda|} \widecheck k_{01}(\lambda)
$$
and
$$
K_{10}(\lambda) = \frac {h(\frac \lambda \epsilon)}{\frac i 4 \sgn \lambda - \frac 1 {2\pi} \log |\lambda|} \widecheck k_{10}(\lambda),
$$
with all the coefficients in $\U_2$. In conclusion, $(U+T(\lambda))^{-1}$ has the form
\be\lb{ut}\begin{aligned}
h(\frac \lambda \epsilon) (U+T(\lambda))^{-1} &= \left(\frac {h(\frac\lambda \epsilon)}{\frac i 4 \sgn \lambda-\frac 1 {2\pi} \log |\lambda| \|V\|_{L^1_x}}+ \frac {h(\frac\lambda \epsilon)\widecheck k_{11}(\lambda)}{(\frac i 4 \sgn \lambda-\frac 1 {2\pi} \log |\lambda|)^2}\right)  P + \\
&+ Q \frac {h(\frac \lambda \epsilon) \widecheck k_{01}(\lambda)}{\frac i 4 \sgn \lambda - \frac 1 {2\pi} \log |\lambda|} P + P \frac {h(\frac \lambda \epsilon) \widecheck k_{10}(\lambda)}{\frac i 4 \sgn \lambda - \frac 1 {2\pi} \log |\lambda|} Q + Q \widecheck k_{00}(\lambda) Q.
\end{aligned}\ee

By Lemma \ref{free_resolvent} we get that
\be\lb{r00v}
R_0((\lambda+i0)^2) v = \widecheck L_1(\lambda) + h(\lambda)(\frac i 4 \sgn \lambda - \frac 1 {2\pi} \log |\lambda|) 1 \otimes v.
\ee
Here $L_1 \in \U_{L^2_x, (1+\log_+|x|) L^\infty_x}$ and $\one_{[0, 1]}(t) L_1 \in \U_{L^2_x, L^\infty_x}$. Likewise
\be\lb{vr}
v R_0((\lambda+i0)^2) = \widecheck L_2(\lambda) + h(\lambda)(\frac i 4 \sgn \lambda - \frac 1 {2\pi} \log |\lambda|) v \otimes 1,
\ee
where $L_2 \in \U_{(1+\log_+|x|)^{-1} L^1_x, L^2_x}$ and $\one_{[0, 1]}(t) L_2 \in \U_{L^1_x, L^2_x}$.

Expanding out the second term in (\ref{sym_id}) in accordance with (\ref{ut}), (\ref{r00v}), and (\ref{vr}) and taking into account the fact that $Qv=0$, we obtain that at low energy
\be\lb{technical_formula}\begin{aligned}
&h(\frac\lambda\epsilon) R_0((\lambda+i0)^2) v (U+T(\lambda))^{-1} v R_0((\lambda+i0)^2) = \\
&= h(\frac\lambda\epsilon)(\frac i 4 \sgn \lambda - \frac 1 {2\pi} \log|\lambda|) 1 \otimes 1 + \widecheck K(\lambda),
\end{aligned}\ee
where $K \in \U_{(1+\log_+|x|)^{-1} L^1_x, (1+\log_+|x|) L^\infty_x}$.

Using a partition of unity adapted to our three cases (high, medium, and low energy), following our previous analysis we can extend (\ref{technical_formula}) to all energies. By Lemma \ref{free_resolvent} the main term $h(\frac\lambda\epsilon)(\frac i 4 \sgn \lambda - \frac 1 {2\pi} \log|\lambda|) 1 \otimes 1$ cancels in (\ref{sym_id}) and we obtain the third claimed property for $R_V$ in (\ref{rv}), namely that $\one_{[1, \infty)}(t)  (R_V(\lambda+i0)^2)^\wedge(t) \in \U_{(1+\log_+|x|)^{-1} L^1_x, (1+\log_+|x|) L^\infty_x}$.

In addition, it is easy to show (see Lemma 2.5 in \cite{bec_new_schroedinger}) that, because $(U+T(\lambda))^{-1}$ is bounded in the upper half-plane and its Fourier transform is integrable, its Fourier transform is supported on $[0, \infty)$. Together with the above analysis, this proves the first two parts of (\ref{rv}).

The statement (\ref{r0v}) also follows from (\ref{assym_id}), (\ref{ut}), and Lemma \ref{free_resolvent}. By keeping more carefully track of the factors and projections involved --- note that
$$
Qv = v (I-\frac {1 \otimes |V|}{\|V\|_{L^1_x}})
$$
--- we obtain (\ref{r0v'}) instead.
\end{proof}

We also state a lemma concerning the equivalence of certain norms defined using $H$ and the usual Sobolev norms (see Lemma 13 in \cite{becgol}):
\begin{lemma}\lb{norm_equivalence} Assume that $V \in (1+\log_+ |x|)^{-2k} L^1_x$ ($k \geq 1$), for some $q >1$ $V \in L^q_{loc}$ and $\lim_{R \to \infty} R^{1-1/q} \|V\|_{L^q(|x| \in [R, 2R])} = 0$, and that zero is not a singular value for $H=-\Delta+V$. Then
\be\lb{1st_conclusion}
\|Hf\|_{(1+\log_+|x|)^{-k} L^1_x} \sim \|-\Delta f\|_{(1+\log_+|x|)^{-k} L^1_x}.
\ee
Moreover, assume $V \in L^q_x$ for $q>1$. Then
\be\lb{2nd_conclusion}
\||H|^{s/2} f\|_{L^2_x} \les \|f\|_{H^s_x}
\ee
for $0 \leq s \leq 1$. Finally,
\be\lb{3rd_conclusion}
\||H|^{-s/2} f\|_{L^2_x} \les \|f\|_{\dot H^{-s-}_x} + \|f\|_{\dot H^{-s+}_x}
\ee
for $0 \leq s \leq 1$.

By duality we also obtain, for $0 \leq s \leq 1$,
$$
\||H|^{s/2} f\|_{H^{-s}_x} \les \|f\|_{L^2_x},\ \|f\|_{\dot H^{s+}_x+\dot H^{s-}_x} \les \||H|^{s/2} f\|_{L^2_x}.
$$
\end{lemma}
These estimates' lack of sharpness is tied to the failure of the endpoint Sobolev embedding $\dot H^1 \subset L^\infty$ in dimension two.
\begin{proof} The first statement (\ref{1st_conclusion}) is a consequence of the boundedness of $I+VR_0(0)$ and $(I+VR_0(0))^{-1}$ as operators on $(1+\log_+|x|)^{-k}L^1_x$. The latter boundedness follows because
$$
(I+VR_0(0))^{-1} = I - v(U+T(0))^{-1}vR_0(0),
$$
where $vR_0(0) \in B((1+\log_+|x|)^{-1} L^1_x, L^2_x)$ and we proved that $(U+T(0))^{-1}$ is $L^2$-bounded in the course of proving Lemma \ref{technical_lemma}.

For the second statement, note that since $V \in L^q$, $V$ is $(-\Delta+1)$-form bounded. Since the eigenfunctions of $-\Delta + V$ (corresponding to negative eigenvalues) are in $H^1$, due to exponential decay they are in $L^1_x \cap L^p_x$ for any $p<\infty$. Therefore $|H|P_p$ is also $-\Delta+1$-form bounded, so $|H|=-\Delta+V+2|H|P_p$ is $-\Delta+1$-form bounded. Thus $\||H|^{1/2} f\| \les \|f\|_{H^1}$. The second statement (\ref{2nd_conclusion}) follows by interpolation with the $s=0$ case.

The third statement (\ref{3rd_conclusion}) is proved once we show that $\|H^{-1/2} f\|_{L^2_x} \les \|(|\nabla|^{-1-\epsilon} + |\nabla|^{-1+\epsilon}) f\|_{L^2_x}$; the projection on the point spectrum is bounded since eigenstates are in $H^1$. Equivalently we have $\|H^{-1/2} |\nabla|^2 (|\nabla|^{1+\epsilon}+|\nabla|^{1-\epsilon})^{-1} f\|_{L^2_x} \les \|f\|_{L^2_x}$. By the $T T^*$ method it suffices to show that
$$
\|(|\nabla|^{1+\epsilon}+|\nabla|^{1-\epsilon})^{-1} |\nabla|^2 H^{-1} |\nabla|^2 (|\nabla|^{1+\epsilon}+|\nabla|^{1-\epsilon})^{-1} f\|_{L^2_x} \les \|f\|_{L^2_x}.
$$
This follows immediately by using the formula (\ref{sym_id}) for $H^{-1}=R_V(0)$, expressing $(|\nabla|^{1+\epsilon}+|\nabla|^{1-\epsilon})^{-1}$ as a convolution with an $L^2_x$ function, and noting again that $(U+T(0))^{-1}$ is $L^2$-bounded.
\end{proof}

We continue with the proof of the main result of this paper. We prove Theorem \ref{dispersive} and Lemma \ref{lema} together.

\begin{proof}[Proofs of Theorem \ref{dispersive} and Lemma \ref{lema}] By spectral calculus, for $f \in L^2$
$$
\frac{\sin(t\sqrt H) P_c}{\sqrt H} f = \frac 1 {2\pi i} \int_0^\infty \frac {\sin(t\sqrt \lambda)}{\sqrt\lambda} (R_V^+(\lambda)-R_V^-(\lambda)) f \dd \lambda.
$$
Making a change of variable from $\lambda$ to $\lambda^2$, we obtain
\be\lb{sin}
\frac{\sin(t\sqrt H) P_c}{\sqrt H} f = \frac 1 {\pi i} \int_{-\infty}^\infty \sin(t \lambda) R_V((\lambda+i0)^2) f \dd \lambda.
\ee
This is the symmetric part of the Fourier transform of $R_V((\lambda+i0)^2) f$. By Lemma \ref{technical_lemma}, $R_V$ has the properties given by (\ref{rv}). The conclusions (\ref{c1}-\ref{c3}) now follow by the definition of the spaces $\U$.

Concerning the cosine evolution, one likewise has that
\be\lb{cos}\begin{aligned}
\cos(t\sqrt H) P_c &= \frac 1 {\pi i} \int_{-\infty}^\infty \lambda \cos(t\lambda) R_V((\lambda+i0)^2) \dd \lambda \\
&= \frac 1 {\pi i} \int_{-\infty}^\infty \lambda \cos (t\lambda) (I+R_0((\lambda+i0)^2)V)^{-1} R_0((\lambda+i0)^2).
\end{aligned}\ee
Since the Fourier transform of the free resolvent $R_0((\lambda+i0)^2)$ is $\one_{|x-y| \leq t} (t^2-|x-y|^2)^{-1/2} = 2\pi \one_{t \geq 0} \frac {\sin(t\sqrt{-\Delta})}{\sqrt{-\Delta}}$ (see Lemma \ref{free_resolvent}), it follows that the Fourier transform of $\lambda R_0((\lambda+i0)^2)$ is $2\pi i \one_{t \geq 0} \cos(t\sqrt{-\Delta})$. By formula (\ref{est_cos}) we obtain that
$$
\|\cos(t\sqrt{-\Delta}) f\|_{L^\infty_x L^1_t} \les \|\nabla f\|_{L^1_x},
$$
which suffices together with (\ref{r0v}) to prove (\ref{c40}-\ref{c5}).


Next, we prove the inhomogenous Strichartz-type inequality (\ref{c6}). We start from the observation that
$$
\left\|\one_{t \geq r} \frac 1 {\sqrt{t^2-r^2}}\right\|_{L^p_t} = C_p r^{-1+1/p}
$$
for $1<p<2$. In fact, it is also the case that $\one_{t \geq r} \frac 1 {\sqrt{t^2-r^2}} \in L^{1, \infty}_t \cap L^{2, \infty}_t$, with norms equal to $C$, respectively $Cr^{-1/2}$. Also note that $\||x|^{-1+1/p}\|_{L^{2p/(p-1), \infty}_x}< \infty$. Then in the free case
$$\begin{aligned}
\left\|\int_{-\infty}^t \frac{\sin(t-s)\sqrt{-\Delta}}{\sqrt{-\Delta}} F(x, s) \dd s\right\|_{L^{q_1, \sigma}_x L^{r_1, \tilde \sigma}_t} &\les \left\|\int \frac {\one_{t-s \geq |x-y|}} {\sqrt{(t-s)^2-|x-y|^2}} F(y, s) \dd s \dd y\right\|_{L^{q_1, \sigma}_x L^{r_1, \tilde \sigma}_t}\\
&\les \||x|^{-1+1/p} \ast \|F(x, s)\|_{L^{q_1, \sigma}_x L^{r_2, \tilde \sigma}_s} \\
& \les \|F(x, s)\|_{L^{q_2, \sigma}_x L^{r_2, \tilde \sigma}_s},
\end{aligned}$$
where by Young's inequality
\be\lb{Young}
\frac 1 {q_1} = \frac 1 {q_2} + \frac {p-1}{2p} - 1,\ \frac 1 {r_1} = \frac 1 {r_2} + \frac 1 p - 1,
\ee
and $1<p \leq 2$, $1\leq q_1, q_2, r_1, r_2 \leq \infty$, with certain modifications in the endpoint cases. Putting the two relations (\ref{Young})  together one obtains the scaling condition that
\be\lb{scaling_cond}
\frac 2 {q_1} + \frac 1 {r_1} + 2 = \frac 2 {q_2} + \frac 1 {r_2}.
\ee
Also we note the following restriction on the range of admissible exponents: $0<\frac 1 {r_2} - \frac 1 {r_1}\leq \frac 1 2$.

In the perturbed case one also has the Duhamel term
$$
\int_{-\infty}^\infty \sin(t\lambda) R_0((\lambda+i0)^2) v (U+T(\lambda))^{-1} v R_0((\lambda+i0)^2) \dd \lambda.
$$
This corresponds to a sequence of operators
$$\begin{aligned}
L^{q_2, 1}_x L^{r_2, \tilde \sigma}_t \xrightarrow{\one_{t \geq 0} \frac{\sin(t\sqrt{-\Delta})}{\sqrt{-\Delta}}\ast} L^\infty_x L^{r, \tilde \sigma}_t \xrightarrow{v} L^2_x L^{r, \tilde \sigma}_t \xrightarrow{((U+T(\lambda))^{-1})^\wedge(t) \ast} L^2_x L^{r, \tilde \sigma}_t \xrightarrow{v} L^1_x L^{r, \tilde \sigma}_t \\\xrightarrow{\one_{t \geq 0} \frac{\sin(t\sqrt{-\Delta})}{\sqrt{-\Delta}}\ast} L^{q_1, \infty}_x L^{r_1, \tilde \sigma}_t.
\end{aligned}$$
The conditions on the exponents are
$$
\frac 2 {q_2} + \frac 1 {r_2} = \frac 1 r + 2 = \frac 2 {q_1} + \frac 1 {r_1} + 2.
$$
In the endpoint case $q_2=4/3$ and $r_2=2$ one has $L^{r_2, \tilde \sigma}=L^{2, 1}$ instead. Interpolation allows us to replace $L^{q_1, \infty}$ and $L^{q_2, 1}$ by $L^{q_1, \sigma}$ and $L^{q_2, \sigma}$ everywhere except at the endpoints.

We note that this is the same scaling relation as in the free case and the admissible range of exponents is at least as wide as in the free case. This concludes the proof of (\ref{c6}).

For the cosine propagator, taking into account (\ref{est_cos}), we need to consider the effect of convolving with $g(r, t):=\one_{[0, r]}(t) \frac 1 r + \one_{[r, \infty)}(t) \frac r {(t+\sqrt{t^2-r^2})\sqrt{t^2-r^2}}$. Note that when $1\leq p<2$ $\|g(r, t)\|_{L^p_t} = C_p r^{-1+1/p}$ and $\|g(r, t)\|_{L^{2, \infty}} = C r^{-1/2}$. The proof of (\ref{c7}) proceeds from here in the same manner as that of (\ref{c6}) for the sine propagator.

We now turn to the proof of the homogenous Strichartz estimates, starting with (\ref{c8}). Following \cite{becgol}, we use the $T T^*$ method, noting that
\be\lb{sincos}
\frac {\sin(t\sqrt H)P_c}{\sqrt H} \frac {\sin(s\sqrt H)P_c}{\sqrt H} = \frac 1 2 \left(\frac{\cos((t-s)\sqrt H)P_c} H - \frac{\cos((t+s)\sqrt H)P_c} H\right).
\ee
By the spectral representation formula we obtain that
\be\lb{cosh}
\frac{\cos(t\sqrt H) P_c} H = \frac 1 {\pi i} \int_{-\infty}^\infty \cos(t\lambda) (I+R_0((\lambda+i0)^2)V)^{-1} \frac {R_0((\lambda+i0)^2)}{\lambda} \dd \lambda.
\ee
The Fourier transform $(\frac{R_0((\lambda+i0)^2)}{\lambda})^\wedge(t)$ is zero for $t <0$ by contour integration and, modulo a contour integral equal to zero, is the same for $t \geq 0$ as
$$
2 \int_{-\infty}^\infty \cos(t\lambda) \frac {R_0((\lambda+i0)^2)}\lambda \dd \lambda = \int_0^\infty \cos(t\sqrt\lambda) \frac {R_0^+(\lambda)-R_0^-(\lambda)} \lambda \dd\lambda = \frac{\cos(t\sqrt{-\Delta})}{-\Delta}.
$$
We explicitly represent $\frac{\cos(t\sqrt {-\Delta})} {-\Delta}$ by noting that
$$\begin{aligned}
\frac{\cos(t\sqrt {-\Delta})} {-\Delta} &= \frac 1 {-\Delta} - \int_0^t \frac{\sin(s\sqrt{-\Delta})}{\sqrt{-\Delta}} \dd s \\
&= -\frac 1 {2\pi} \log|x-y| - \frac 1 {2\pi} \int_{|x-y|}^{|x-y|\vee t} \frac {ds} {\sqrt{s^2-|x-y|^2}}.
\end{aligned}$$
Therefore
$$
\left|\frac{\cos(t\sqrt {-\Delta})} {-\Delta}\right|(x, y) \les 1 + |\log\max(|x-y|, t)|,
$$
so
$$
\sup_{s \in [0, t]} \left|\frac{\cos(s\sqrt {-\Delta})} {-\Delta}\right|(x, y) \les 1 + |\log|x-y|| + \log_+ t.
$$
However, we require a finer estimate: in fact
$$
\sup_{t \in [0, \infty)} \left|\frac{\cos(t\sqrt{-\Delta})}{-\Delta}(x, y) - \log_+ t\right| \les 1 + |\log|x-y||,
$$
where $\log_+ t$ can be replaced by a mollified version. We have isolated the worst-behaved term as a rank-one operator $(\log_+ t) 1 \otimes 1$.

Plugging this into (\ref{cosh}) and using the fact that $(I-\frac {1 \otimes |V|}{\|V\|_{L^1_x}}) 1 = 0$, we obtain that the term $\log_+ t$ only meaningfully interacts with the third term in formula (\ref{r0v'}), namely $\frac{h(\lambda)}{\frac i 4 \sgn \lambda - \frac 1 {2\pi} \log|\lambda|} \widecheck N_2(\lambda)$. A simple computation shows that the convolution product of these two expressions is uniformly bounded (it has Fourier transform $\frac c \lambda$). Therefore
$$\begin{aligned}
\left\|\int_{-\infty}^t \frac{\cos((t-s)\sqrt H)P_c} H F(s) \dd s \right\|_{(1+\log_+|x|)L^\infty_x L^\infty_t} \les \\
\les \|F\|_{((1+\log_+|x|)^{-1}L^1_x \cap \mc K_1) L^1_s}.
\end{aligned}$$
Finally, using (\ref{sincos}) we get (\ref{c8}).

The inequality (\ref{c9}) is proved in the same way, noting that
$$
\frac {\cos(t\sqrt H)P_c}{\sqrt H} \frac {\cos(s\sqrt H)P_c}{\sqrt H} = \frac 1 2 \left(\frac{\cos((t-s)\sqrt H)P_c} H + \frac{\cos((t+s)\sqrt H)P_c} H\right)
$$
and then making use of the comparison Lemma \ref{norm_equivalence} at the end to introduce the Sobolev norm.

Next, we prove the nonendpoint Strichartz estimates (\ref{c11}) and (\ref{c12}). We start from the operators $\frac{\sin(t\sqrt H)P_c}{H^{s/2}}$ and $\frac{\cos(t\sqrt H)P_c}{H^{s/2}}$, $1/4 < s < 1$. By the same $T T^*$ argument as above, the problem reduces to examining the operators
\be\lb{nonendpoint}
\frac{\cos(t\sqrt H)P_c}{H^s} = \int_{-\infty}^\infty \cos(t\lambda) (I + R_0((\lambda+i0)^2)V)^{-1} \frac{R_0((\lambda+i0)^2)}{|\lambda|^{2s-1}\sgn\lambda} \dd \lambda.
\ee
We are compelled to study the Fourier transform $M(t):=\displaystyle \left(\frac{R_0((\lambda+i0)^2)}{|\lambda|^{2s-1}\sgn\lambda}\right)^\wedge(t)$.
We cannot use contour integration since the function is not analytic. The Fourier transform is explicitly given by
$$
M(t)(x, y) = \frac i 4 \left(\int_0^\infty e^{-it\lambda} H_0^+(\lambda|x-y|) \lambda^{1-2s} \dd \lambda - \int_0^\infty e^{it\lambda} H_0^-(\lambda|x-y|) \lambda^{1-2s} \dd \lambda\right).
$$
By rescaling this becomes
$$
M(t)(x, y)=|x-y|^{2s-2} \frac i 4 \left(\int_0^\infty e^{-i\frac t {|x-y|} \lambda} H_0^+(\lambda) \lambda^{1-2s} \dd \lambda - \int_0^\infty e^{i\frac t {|x-y|}\lambda} H_0^-(\lambda) \lambda^{1-2s} \dd \lambda\right)
$$
Recall that $H_0^\pm(z) = \pm J_0(z) + iY_0(z)$, where $J_0$ and $Y_0$ are the Bessel functions of first and second kind. We have the following asymptotic expansions, see \cite{green}:
$$
H_0^\pm(z) = \pm 1 - \frac {2i}\pi \log(z/2) - \frac{2i\gamma}\pi + \tilde O(z^2\log z)
$$
for $|z|<1$ and
$$
H_0^\pm(z) = e^{\pm iz} \omega_\pm(z),\ \omega_\pm(z) = \tilde O((1+|z|)^{-1/2})
$$
for $|z|>1$. For simplicity we used the notation $f=\tilde O(g)$ if $f^{(k)} = O(g^{(k)})$ for all $k\geq 0$. Recall $h$ is a smooth cutoff function. Then by stationary phase, if we set $k= \frac t {|x-y|}$,
$$
\left|\int_0^\infty e^{ik\lambda} (1+\log\lambda) \lambda^{1-2s} h(\lambda) \dd \lambda\right| \les \min(1, \frac {1+\log k}{k^{2-2s}})
$$
with better bounds for the error term and
$$
\left|\int_0^\infty e^{i(k-1)\lambda} \lambda^{(1/2)-2s} (1-h(\lambda)) \dd \lambda\right| \les \min(|k-1|^{-N}, |k-1|^{2s-3/2}),
$$
where the last term is replaced by $1+\log_- |k-1|$ if $s=3/4$ and $1$ if $s>3/4$.

Thus we obtain that
\be\lb{alternative}
M(t)(x, y) \les \left\{\begin{array}{ll} \displaystyle\frac{\log \frac t {|x-y|}}{t^{2-2s}},& t >> |x-y|,\\
\displaystyle \frac 1 {|x-y|^{2-2s}} + |x-y|^{-1/2}(t-|x-y|)^{2s-3/2},& t \sim |x-y|,
\end{array}\right.\ee
where the last term gets replaced by $|x-y|^{-1/2} \log_- \left|\frac{t-|x-y|}{|x-y|}\right|$ when $s=3/4$ and by $0$ when $s>3/4$.

Note that this expression is not locally integrable in $t$ for $s \leq 1/4$. For $1/4<s<1/2$,  $M(t)(x, y) \in L^p_t$, $1 \leq p < \frac 2 {3-4s}$, and $M(t)(x, y) \in L^{\frac 2 {3-4s}, \infty}$. For $1/2 \leq s \leq 3/4$, $M(t)(x, y) \in L^p_t$ for $\frac 1 {2-2s} < p < \frac 2 {3-4s}$ and $M(t)(x, y) \in L^{\frac 2 {3-4s}, \infty}$ except when $s=3/4$. Finally, for $3/4<s<1$ $M(t)(x, y) \in L^p_t$ for $\frac 1 {2-2s} < p \leq \infty$. One always has $\displaystyle \|M(t)(x, y)\|_{L^p_t} = \frac C {|x-y|^{2-2s-1/p}}$ due to scaling.

Therefore the operator with convolution kernel $M(t)(x, y)$ takes $L^{q_2, \sigma}_x L^{r_2, \tilde \sigma}_t$ to $L^{q_1, \sigma}_x L^{r_1, \tilde \sigma}_t$, where by Young's inequality
$$
\frac 1 {q_1} = \frac 1 {q_2} -s-\frac 1{2p},\ \frac 1 {r_1} = \frac 1 {r_2} + \frac 1 p - 1,
$$
with the natural modifications at the endpoints. Combining the two exponent relations we obtain the scaling condition
\be\lb{scaling_cond2}
\frac 2 {q_1} + \frac 1 {r_1} + 2s + 1 = \frac 2 {q_2} + \frac 1 {r_2}.
\ee
This is accompanied by various restrictions on the range of allowable exponents (see the statement of Lemma \ref{lema}).

In the free case this analysis is sufficient, but in the perturbed case, taking into account (\ref{nonendpoint}) and (\ref{assym_id}), we also have to consider the term
$$
\int_{-\infty}^\infty \cos(t\lambda) R_0((\lambda+i0)^2) v (U+T(\lambda))^{-1} v \widecheck M(\lambda) \dd \lambda.
$$
This corresponds to a chain of operators
$$\begin{aligned}
L^{q_2, 1}_x L^{r_2, \tilde \sigma}_t \xrightarrow{M(t) \ast} L^\infty_x L^{r, \tilde \sigma}_t \xrightarrow{v} L^2_x L^{r, \tilde \sigma}_t \xrightarrow{((U+T(\lambda))^{-1})^\wedge(t) \ast} L^2_x L^{r, \tilde \sigma}_t \xrightarrow{v} L^1_x L^{r, \tilde \sigma}_t \\
\xrightarrow{\one_{t \geq 0} \frac{\sin(t\sqrt{-\Delta})}{\sqrt{-\Delta}}\ast} L^{q_1, \infty}_x L^{r_1, \tilde \sigma}_t.
\end{aligned}$$
The exponents satisfy the scaling conditions (\ref{scaling_cond}) and (\ref{scaling_cond2}), that is $\frac 2 {q_2} + \frac 1 {r_2} = 2 + \frac 1 r  = \frac 2 {q_1} + \frac 1 {r_1} + 2s + 1$. We can replace the Lorentz spaces $L^{q_1, \infty}$ and $L^{q_2, 1}$ by the Lorentz spaces $L^{q_1, \sigma}$ and $L^{q_2, \sigma}$ by interpolation, except at the endpoints. Note that the scaling condition is the same as in the free case and the admissible range of exponents is at least as wide.

Exactly the same proof applies to the case of the sine evolution. This finishes the proof of (\ref{c10}), hence of Lemma \ref{lema}.

Now, in order to prove (\ref{c11}) and (\ref{c12}), we set the exponents $q_1$ and $q_2$ and $r_1$ and $r_2$ to be dual to each other in the above and $\sigma=\tilde \sigma=2$. The scaling condition becomes
$$
\frac 2 {q_1} + \frac 1 {r_1} = 1-s
$$
and the restriction on the range of available exponents becomes $2 \leq r_1 \leq \frac 1 {3/4-s}$ for $1/4<s<1/2$, $\frac 1 {1-s} < r_1 \leq \frac 1 {3/4-s}$ for $1/2\leq s \leq 3/4$ (except no equal sign when $s=\frac 3 4$), and $\frac 1 {1-s}<r_1 \leq \infty$ for $3/4<s<1$. Several of these conditions can be represented more efficiently as $8 \leq q_1<\infty$ and $2 \leq r_1 \leq \infty$. When $r_1=\infty$ we replace $L^{r_1, 2}$ by $L^\infty$.

Note that in the case of the free Laplacian, due to the cancellation of the logarithmic term at low frequencies, one obtains for $\frac{\cos(t\sqrt{-\Delta})}{(-\Delta)^{s/2}}$ an estimate similar to (\ref{alternative}), but without the logarithmic factor. The absence of the logarithmic factor implies the validity of the endpoint estimate $M(t)(x, y) \in L^{\frac 1 {2-2s}, \infty}$ for $\frac 1 2 \leq s < 1$. Therefore one can take $q_1=\infty$ in the free case (but not if $s=\frac 1 2$ and $r_1=2$, since that would correspond to convolving with $t^{-1}$).

We then use a $TT^*$ argument as in the proof of the homogenous Strichartz estimates (\ref{c8}) and (\ref{c9}) to complete the proofs of (\ref{c11}) and (\ref{c12}).
\end{proof}

Next, we prove Proposition \ref{other_estimates}.

\begin{proof}[Proof of Proposition \ref{other_estimates}] By following the same approach as in the proof of Lemma \ref{technical_lemma}, under the stronger assumption on $V$ we obtain in addition to (\ref{rv}) that
$$
R_V((\lambda+i0)^2) = \widecheck Q_0(\lambda) + \frac{h(\lambda)}{\log|\lambda|} \widecheck Q_1(\lambda),
$$
where $\one_{[1, \infty)}(t) (1+\log_+|t|) Q_0,\ Q_1\in \U_{(1+\log_+|x|)^{-2}L^1_x, (1+\log_+|x|)^2L^\infty_x}$.
In addition, recall that the Fourier transform of $\frac{h(\lambda)}{\log|\lambda|}$ is of size $\frac 1 {t \log^2 |t|}$ for large $t$. Estimate (\ref{d1}) follows immediately.

Note that
$$
\frac{\cos(t\sqrt H)P_c}H -\int_t^T \frac{\sin(\tau\sqrt H)P_c}{\sqrt H} \dd \tau = \frac{\cos(T\sqrt H)P_c}H,
$$
which converges weakly to zero as $T \to \infty$. Therefore
$$
\frac{\cos(t\sqrt H)P_c}H = \int_t^\infty \frac{\sin(\tau\sqrt H)P_c}{\sqrt H} \dd \tau.
$$
Then (\ref{d2}) is a simple integration, followed by the use of the norm equivalence Lemma \ref{norm_equivalence}.

Next, again assuming that $V \in (1+\log_+|x|)^{-4} L^1_x$, instead of (\ref{r0v}) we more precisely get that
$$
(I+R_0((\lambda+i0)^2) V)^{-1} = \widecheck Q_2(\lambda) + \frac{h(\lambda)}{\log|\lambda|} \widecheck Q_3(\lambda),
$$
where $Q_2, Q_3 \in (1+\log_+|t|)^{-1} \U_{(1+\log_+|x|)^{-2}L^1_x, (1+\log_+|x|)^2L^\infty_x}$. Hence (\ref{d3}) follows once we note that by (\ref{est_cos})
$$
\|(1 + \log_+|t|)^k \cos(t\sqrt{-\Delta}) f\|_{(1+\log_+|x|)^k L^\infty_x L^1_t} \les \|\nabla f\|_{(1+\log_+|x|)^{-k} L^1_x}.
$$
Then (\ref{d4}) is a simple consequence of integration.
\end{proof}

\section{Proof of the decay estimate}
\begin{proof}[Proof of Theorem \ref{decay_estimate}]
We again separate the proof into three parts --- the high frequency, the medium frequency, and the low frequency part. We use two Banach spaces in the proof, namely $L^1_t$ and $t^{-1/2} L^\infty_t$. Note that
$$
\|f \ast g\|_{L^1_t} \les \|f\|_{L^1_t} \|g\|_{L^1_t},\ \|f \ast g\|_{t^{-1/2} L^\infty_t} \les \|f\|_{t^{-1/2} L^\infty_t} \|g\|_{L^1_t} + \|f\|_{L^1_t} \|g\|_{t^{-1/2} L^\infty_t},
$$
and
$$
\|\one_{t \geq 0} t^{-1/2} \ast \one_{t \geq 0} t^{-1/2}\|_{L^\infty_t} < \infty.
$$
Our proof strategy will be to show that the integral kernel we are examining belongs to both $L^1_t$ and $t^{-1/2} L^\infty_t$. Note that $\widehat h(t)$ and other Fourier transformed cutoff functions belong to both spaces, so cutting off in frequency preserves both spaces.

High energy: We represent the high frequency part of the evolution as
\be\lb{high_energy}
\frac 1 {\pi i} \int_{-\infty}^\infty e^{-it\lambda} R_0((\lambda+i0)^2) (I + V R_0((\lambda+i0)^2))^{-1} (1-h(\lambda/R)) f \dd \lambda.
\ee
What we show is that for sufficiently high $a$
$$
\left\|\int_{-\infty}^\infty e^{-it\lambda} R_0((\lambda+i0)^2) ((I + V R_0((\lambda+i0)^2))^{-1}-I) (1-h(\lambda/R)) f \dd \lambda \right\|_{L^\infty_x} \les |t|^{-1/2} \|f\|_{L^1_x}.
$$
The desired estimate follows for $H$ knowing that it is true for $-\Delta$, i.e.\ knowing that (see e.g.\ \cite{beals} for the proof)
$$
\|e^{it\sqrt{-\Delta}} \langle \Delta \rangle^{-3/4-} f\|_{L^\infty_x} \les |t|^{-1/2} \|f\|_{L^1_x},\ \|e^{it\sqrt{-\Delta}} \langle \Delta \rangle^{-3/4} f\|_{L^\infty_x} \les |t|^{-1/2} \|f\|_{\mc H^1_x}
$$
and same for $\Delta^{-1/2} \langle \Delta \rangle^{-1/4}$.

Consider the expression $\one_{t \geq r} (t^2-r^2)^{-1/2}$. Note that $\|\one_{t \geq r} (t^2-r^2)^{-1/2}\|_{L^1} = +\infty$ and that $\|\one_{t \geq r} (t^2-r^2)^{-1/2}\|_{(t-r)^{-1/2}L^\infty} \les r^{-1/2}$. However, for the high frequency part we have the improved bound
\be\lb{detailed_computation}\begin{aligned}
&\|(\delta_0(t) - R \widehat h(Rt)) \ast \one_{t \geq r} (t^2-r^2)^{-1/2}\|_{L^1_t} \les \\
&\les \int_r^{r+\rho} (t^2-r^2)^{-1/2} \dd t + R^{-1} \int_{r+\rho}^\infty \frac d {dt} (t^2-r^2)^{-1/2} \dd t \\
&\les \min(\rho^{1/2} r^{-1/2}, 1+\log_+(\rho/r)) + R^{-1} \min(r^{-1/2} \rho^{-1/2}, \rho^{-1}) \\
&\les \min(R^{-1/2} r^{-1/2}, 1+\log_-(Rr)),
\end{aligned}\ee
where we set $\rho = R^{-1}$. Also, 
$$
\|R \widehat h(Rt) \ast \one_{t \geq r} (t^-r^2)^{-1/2}\|_{(t-r)^{-1/2} L^\infty_t} \les r^{-1/2} \|R \widehat h(Rt) \ast \one_{t \geq r} (t-r)^{-1/2}\|_{(t-r)^{-1/2} L^\infty_t}
$$
and uniformly for all $R$
$$
\|R \widehat h(Rt) \ast \one_{t \geq 0} t^{-1/2}\|_{t^{-1/2} L^\infty_t} \les \|\one_{(-\infty, t/2]}(s) R \widehat h(Rs)\|_{L^1_s} + t^{1/2} \sup_{s \in [t/2, t]} R h(Rs) \les t^{-1/2}.
$$
Therefore
$$
\|(\delta_0(t) - R \widehat h(Rt)) \ast \one_{t \geq r} (t^2-r^2)^{-1/2}\|_{(t-r)^{-1/2} L^\infty_t} \les r^{-1/2}.
$$


We expand the high energy part (\ref{high_energy}) into a geometric (Born) series. Let us first examine the term (where for convenience we denote $r_1=|x-y|$, $r_2=|y-z|$)
\be\lb{rvr}\begin{aligned}
&\int_{-\infty}^\infty e^{-it\lambda} R_0((\lambda+i0)^2) V R_0((\lambda+i0)^2) (1-h(\lambda/R)) \dd \lambda = \\
&= \int (\delta_0-R \widehat h(Rt)) \ast \left((\delta_0 - (R/2) \widehat h(Rt/2)) \ast (\one_{t \geq r_1} (t^2-r_1^2)^{-1/2}) \right) V(y) \ast \\
&\ast \left((\delta_0 - (R/2) \widehat h(Rt/2)) \ast (\one_{t \geq r_2}  (t^2-r_2^2)^{-1/2}) \right) \dd y.
\end{aligned}\ee
On one hand, since $(t^2-r^2)^{-1/2} \les r^{-1/2} (t-r)^{-1/2}$, this term is bounded by
$$
\int (\one_{t \geq r_1} r_1^{-1/2} (t-r_1)^{-1/2}) |V(y)| \ast (\one_{t \geq r_2} r_2^{-1/2} (t-r_2)^{-1/2}) \dd y,
$$
which, taking into account the fact that $\one_{t \geq 0} t^{-1/2} \ast \one_{t \geq 0} t^{-1/2}$ is uniformly bounded (and same for their mollified versions), is then bounded by
$$
(\ref{rvr}) \les \int r_1^{-1/2} |V(y)| r_2^{-1/2} \dd y \les \min(r_1^{-1/2}, r_2^{-1/2}) \|V\|_{L^1 \cap \tilde {\mc K}}.
$$
On the other hand, denoting
$$\begin{aligned}
f_1 = (\delta_0 - (R/2) \widehat h(Rt/2)) \ast (\one_{t \geq r_1} (t^2-r_1^2)^{-1/2}),\\
f_2 = (\delta_0 - (R/2) \widehat h(Rt/2)) \ast (\one_{t \geq r_2}  (t^2-r_2^2)^{-1/2}),
\end{aligned}$$
we have
$$\begin{aligned}
\|f_1 \ast f_2\|_{(t-r_1-r_2)^{-1/2} L^\infty_t} &\les \|f_1\|_{(t-r_1)^{-1/2}L^\infty_t} \|f_2\|_{L^1_t} + \|f_1\|_{L^1_t} \|f_2\|_{(t-r_2)^{-1/2} L^\infty_t} \\
&\les r_1^{-1/2} (1+ \log_-(R r_2)) + (1 + \log_-(R r_1)) r_2^{-1/2}.
\end{aligned}$$
We obtain a bound of
$$\begin{aligned}
(\ref{rvr}) &\les (t-r_1-r_2)^{-1/2} \int |V(y)| (r_1^{-1/2} (1+ \log_-(R r_2)) + (1 + \log_-(R r_1)) r_2^{-1/2}) \dd y \\
&\les (t-r_1-r_2)^{-1/2} \|V\|_{L^1_x \cap \tilde {\mc K}}.
\end{aligned}$$
It is here and in the similar estimate for the general term that we fully use the condition that $V \in \tilde {\mc K}$. Combining the two bounds we obtain
$$
(\ref{rvr}) \les \min(r_1^{-1/2}, r_2^{-1/2}, (t-r_1-r_2)^{-1/2}) \|V\|_{L^1_x \cap \tilde {\mc K}} \les t^{-1/2} \|V\|_{L^1_x \cap \tilde {\mc K}}.
$$

We now consider the general term, for $n \geq 2$,
\be\lb{rvrvr}\begin{aligned}
&\int_{-\infty}^\infty e^{-it\lambda} R_0((\lambda+i0)^2) (V R_0((\lambda+i0)^2))^{n-1} (1-h(\lambda/R)) \dd \lambda = \\
&= \int (\delta_0-R \widehat h(Rt)) \ast \left((\delta_0 - (R/2) \widehat h(Rt/2)) \ast (\one_{t \geq r_1} (t^2-r_1^2)^{-1/2}) \right) V(y_1) \ast \ldots \\
&\ast V(y_{n-1}) \left((\delta_0 - (R/2) \widehat h(Rt/2)) \ast (\one_{t \geq r_n}  (t^2-r_n^2)^{-1/2}) \right) \dd y_1 \ldots \dd y_{n-1},
\end{aligned}\ee
where $r_1=|x-y_1|$, $r_2=|y_1-y_2|, \ldots, r_n=|y_{n-1}-z|$. Suppose we denote
$$
f_k = (\delta_0 - (R/2) \widehat h(Rt/2)) \ast (\one_{t \geq r_k}  (t^2-r_k^2)^{-1/2}).
$$
We obtain that
$$
\|f_1 \ast \ldots \ast f_n\|_{L^1_t} \les \prod_{k=1}^n \min(R^{-1/2} r_k^{-1/2}, 1 + \log_- (Rr_k)),
$$
$$\begin{aligned}
\|f_1 \ast \ldots \ast f_n\|_{(t-r_1-\ldots-r_n)^{-1/2} L^\infty_t} &\les \prod_{k=1}^n \min(R^{-1/2} r_k^{-1/2}, 1 + \log_- (Rr_k)) \cdot \\
&\cdot \sum_{k=1}^n \frac 1 {r_k^{1/2} \min(R^{-1/2} r_k^{-1/2}, 1 + \log_- (Rr_k))},
\end{aligned}$$
and
$$\begin{aligned}
\|f_1 \ast \ldots \ast f_n\|_{L^\infty_t} &\les \prod_{k=1}^{n-1} \min(R^{-1/2} r_k^{-1/2}, 1 + \log_- (Rr_k)) \cdot \\
&\cdot \sum_{k=1}^{n-1} \frac 1 {r_k^{1/2} \min(R^{-1/2} r_k^{-1/2}, 1 + \log_- (Rr_k))} \cdot r_n^{-1/2}.
\end{aligned}$$
Therefore
$$\begin{aligned}
(\ref{rvrvr}) &\les \min((t-r_1-\ldots-r_n)^{-1/2}, r_1^{-1/2}, \ldots, r_n^{-1/2}) R^{2-n} \|V\|^{n-1}_{L^1_x \cap \tilde{\mc K}} \\
&\les n^{1/2} t^{-1/2} R^{2-n} \|V\|^{n-1}_{L^1_x \cap \tilde{\mc K}}.
\end{aligned}$$
For sufficiently large $R$ or small $V$ the series is then summable.

Medium frequencies: We want to show that for every $\lambda_0 \ne 0$ there exists $\epsilon>0$ such that
\be\lb{medium_freq}
\bigg\|\frac 1 {\pi i} \int_{-\infty}^\infty e^{it\lambda} R_0((\lambda+i0)^2) ((I + V R_0((\lambda+i0)^2))^{-1}-I) h(\frac {\lambda-\lambda_0} \epsilon) f \dd \lambda \bigg\|_{L^\infty_x} \les t^{-1/2} \|f\|_{L^1_x}.
\ee

Without loss of generality, assume $\lambda_0>2$. Then we can write $h(\frac {\lambda-\lambda_0} \epsilon) = (1-h(\lambda)) h(\frac {\lambda-\lambda_0}\epsilon)$. Then note that by (\ref{detailed_computation})
\be\lb{L^1_bound}\begin{aligned}
&\|h(\frac{\lambda-\lambda_0}\epsilon)^\wedge(t) \ast (1_{t \geq r} (t^2-r^2)^{-1/2})\|_{L^1_t} \les \\
&\les \|(\delta_0-\widehat h(t)) \ast (1_{t \geq r} (t^2-r^2)^{-1/2})\|_{L^1_t} \les \min (r^{-1/2}, 1+ \log_-r).
\end{aligned}\ee
Furthermore, by a computation similar to (\ref{detailed_computation}) we obtain that
\be\lb{tail}
\int_R^\infty |(\delta_0-\widehat h(t)) \ast (1_{t \geq r} (t^2-r^2)^{-1/2})| \les \min (r^{-1/2}, 1+\log_-r, (R^2-r^2)^{-1/2}).
\ee
At the same time, note that $h(\frac{\lambda-\lambda_0}\epsilon)^\wedge = \epsilon e^{i\lambda_0\epsilon t} \widehat h(\epsilon t)$ and
\be\lb{unu'}
\|(\epsilon e^{i\lambda_0\epsilon t} \widehat h(\epsilon t)) \ast (\one_{t \geq r} (t^2-r^2)^{-1/2})\|_{L^\infty_t} \les \|\epsilon \widehat h(\epsilon t)\|_{L^{2, 1}_t} \|\one_{t \geq r} (t^2-r^2)\|_{L^{2, \infty}_t} \les r^{-1/2} \epsilon^{1/2}.
\ee
On the other hand, $\one_{t \geq r} (t^2-r^2)^{-1/2} \les r^{-1/2} \one_{t \geq r} (t-r)^{-1/2}$ and uniformly in $\epsilon$ $|(\epsilon e^{i\lambda_0\epsilon t} \widehat h(\epsilon t)) \ast t^{-1/2}| \les t^{-1/2}$, so
\be\lb{doi'}
\|(\epsilon e^{i\lambda_0\epsilon t} \widehat h(\epsilon t)) \ast \one_{t \geq r} (t^2-r^2)^{-1/2}\|_{(t-r)^{-1/2} L^\infty_t} \les r^{-1/2}.
\ee
Combining (\ref{unu'}) and (\ref{doi'}) we obtain that
\be\lb{trei}
\|(\epsilon e^{i\lambda_0\epsilon t} \widehat h(\epsilon t)) \ast \one_{t \geq r} (t^2-r^2)^{-1/2}\|_{t^{-1/2} L^\infty_t} \les \epsilon^{1/2}+r^{-1/2}.
\ee

Define the spaces of kernels $\U_1 := \U_{L^1_x, L^1_x \cap \tilde {\mc K}_{1/2}}$ and
$$
\mc V_1 := \{T(t, y, x) \mid \|T(t, y, x)\|_{t^{-1/2}L^\infty_t} \in \B(L^1_x, L^1_x \cap \mc K)\}.
$$
Note that $T \in \mc V_1$ implies that $\|T(t)\|_{\B(L^1_x)} \les t^{-1/2}$. Since
$$
\|f_1 \ast f_2\|_{t^{-1/2} L^\infty_t} \les \|f_1\|_{t^{-1/2}L^\infty_t} \|f_2\|_{L^1_t} + \|f_1\|_{L^1_t} \|f_2\|_{t^{-1/2} L^\infty_t},
$$
we obtain that
$$
\|T_1 \ast T_2\|_{\mc V_1} \les \|T_1\|_{\U_1} \|T_2\|_{\mc V_1} + \|T_1\|_{\mc V_1} \|T_2\|_{\U_1}
$$
and more generally that
\be\lb{subalgebra'}
\bigg\|\prod_{k=1}^n T_k\bigg\|_{\mc V_1} \les C^n \prod_{k=1}^n \|T_k\|_{\U_1} \cdot \sum_{k=1}^n \frac {\|T_k\|_{\mc V_1}}{\|T_k\|_{\U_1}}.
\ee

Let $S_\epsilon(\lambda) = h(\frac {\lambda-\lambda_0} \epsilon) V (R_0((\lambda+i0)^2) - R_0((\lambda_0+i0)^2)$. A simple computation shows that if $V \in L^1_x \cap \tilde {\mc K}$, (\ref{tail}) implies $\lim_{\epsilon \to 0} \|\widehat S_\epsilon\|_{\widehat{\U_1}} = 0$. Furthermore, by (\ref{trei}) $S_{\epsilon} \in \widehat{\mc V_1}$ uniformly for small $\epsilon$ (this also uses $V \in L^1_x \cap \tilde {\mc K}$).

We also observe that by (\ref{assym_id}) and Lemma \ref{invertibility}, if $V \in L^1$ and $V \in L^q_{loc}$, \\
$\lim_{R \to \infty} R^{1-1/q} \|V\|_{L^q(|x| \in [R, 2R])} = 0$, then $(I+VR_0((\lambda_0+i0)^2))^{-1} \in \B(L^1)$. Writing
$$
(I+VR_0((\lambda+i0)^2))^{-1} = I - V R_0((\lambda+i0)^2) (I+VR_0((\lambda+i0)^2))^{-1}
$$
we obtain that, if in addition $V \in \tilde {\mc K}$, then $(I+VR_0((\lambda_0+i0)^2))^{-1} - I \in \B(L^1_x, L^1_x \cap \tilde {\mc K}_{1/2})$ and $(I+VR_0((\lambda_0+i0)^2))^{-1} \in \B(L^1_x \cap \tilde {\mc K}_{1/2})$. We then express
$$\begin{aligned}
&h(\frac {\lambda-\lambda_0} \epsilon) ((I + V R_0((\lambda+i0)^2))^{-1} - I) =\\
&= h(\frac {\lambda-\lambda_0} \epsilon) ((I + V R_0((\lambda_0+i0)^2) + h(\frac {\lambda-\lambda_0} {2\epsilon}) V (R_0((\lambda+i0)^2) - R_0((\lambda_0+i0)^2)))^{-1}-I) \\
&= h(\frac {\lambda-\lambda_0} \epsilon) (\sum_{n=0}^\infty (I + V R_0((\lambda_0+i0)^2))^{-1} (S_{2\epsilon}(\lambda) (I + V R_0((\lambda_0+i0)^2))^{-1})^n - I).
\end{aligned}$$
This expansion converges in $\widehat {\U_1}$ for sufficiently small $\epsilon$. By (\ref{subalgebra'}) we obtain that $h(\frac {\lambda-\lambda_0} \epsilon) ((I + V R_0((\lambda+i0)^2))^{-1} - I) \in \widehat{\mc V_1}$ as well.

In addition, by (\ref{L^1_bound}) and (\ref{trei}),
$$\begin{aligned}
&\|(h(\frac {\lambda-\lambda_0}{2\epsilon}) R_0((\lambda+i0)^2))^\wedge(t)(x, y)\|_{L^1_t} \les 1+\log_-|x-y|, \\
&(h(\frac {\lambda-\lambda_0}{2\epsilon}) R_0((\lambda+i0)^2))^\wedge(t)(x, y) \les t^{-1/2}(1+|x-y|^{-1/2})
\end{aligned}$$
and the conclusion (\ref{medium_freq}) follows.

Low frequencies: Here our goal is to show that for sufficiently small $\epsilon>0$
\be\lb{low_freq}
\left\| \frac 1 {\pi i} \int_{-\infty}^\infty \sin(t\lambda) R_0((\lambda+i0)^2)((I+V R_0((\lambda+i0)^2))^{-1}-I) h(\frac \lambda \epsilon) f \dd \lambda \right\|_{L^\infty_x} \les t^{-1/2} \|f\|_{L^1_x}
\ee
and likewise for the cosine propagator.

Following Lemma \ref{free_resolvent}, we decompose $R_0((\lambda+i0)^2)$ into a component $\widecheck L$ with integrable Fourier transform and a rank-one part:
$$
R_0((\lambda+i0)^2) = \widecheck L(\lambda) + h(\lambda)(\frac i 4 \sgn \lambda - \frac 1 {2\pi} \log|\lambda|) 1 \otimes 1,
$$
where $L(t) = \one_{t \geq r} (t^2-r^2)^{-1/2} - \eta(t) t^{-1} + \widehat{\tilde g}(t)$. 
We know by Lemma \ref{free_resolvent} that
$$
\|\epsilon \widehat h(\epsilon t) \ast L(t)\|_{L^1_t} \les \|L(t)\|_{L^1_t} \les 1 + |\log r|
$$
and that as $R \to \infty$
\be\lb{cond}
\int_R^\infty |L(t)| \dd t \les \chi_{r \geq R} (\log r - \log R) + \max(1, \frac {r^2}{R^2}) + o(1).
\ee
Note that for $1 \leq p \leq 2$ (using Lorentz spaces for $p=2$)
\be\lb{unu}\begin{aligned}
\|\epsilon \widehat h(\epsilon t) \ast \one_{t \geq r} ((t^2-r^2)^{-1/2} - t^{-1})\|_{L^\infty_t} &\les \|\epsilon \widehat h(\epsilon t)\|_{L^{p'}_t} \|\one_{t \geq r} ((t^2-r^2)^{-1/2}) - t^{-1})\|_{L^p_t} \\
&\les r^{-1+1/p} \epsilon^{1/p}.
\end{aligned}\ee

On the other hand, for $0 \leq \alpha \leq \frac 1 2$, $\one_{t \geq r} (t^2-r^2)^{-1/2} \les r^{-\alpha} \one_{t \geq r} (t-r)^{\alpha-1}$ and for $0<\alpha \leq 1$, uniformly in $\epsilon$,
$$
\epsilon \widehat h(\epsilon t) \ast (\one_{t \geq 0} t^{\alpha-1}) \les t^{\alpha-1},
$$
so
\be\lb{doi}
\|\epsilon \widehat h(\epsilon t) \ast (\one_{t \geq r} (t^2-r^2)^{-1/2})\|_{(t-r)^{\alpha-1}L^\infty_t} \les r^{-\alpha}.
\ee
The same remains true after adding the term $\one_{t \geq r} t^{-1}$, which is smaller. Combining (\ref{unu}) and (\ref{doi}) we get that
$$
\|\epsilon \widehat h(\epsilon t) \ast \one_{t \geq r} ((t^2-r^2)^{-1/2} - t^{-1})\|_{t^{-1/2} L^\infty_t} \les \epsilon^{1/2} + r^{-1/2}.
$$
We consider separately the term $\one_{t \geq r} t^{-1} - \eta(t) t^{-1}$, for which we see that
$$
\|\epsilon \widehat h(\epsilon t) \ast \one_{t \geq r} t^{-1} - \eta(t) t^{-1}\|_{t^{-1/2} L^\infty_t} \les \|\one_{t \geq r} t^{-1} - \eta(t) t^{-1}\|_{t^{-1/2} L^\infty_t} \les 1+r^{-1/2}.
$$
Therefore
$$
\|\epsilon \widehat h(\epsilon t) \ast L(t)\|_{t^{-1/2} L^\infty_t} = \|\epsilon \widehat h(\epsilon t) \ast (\one_{t \geq r} (t^2-r^2)^{-1/2} - \eta(t) t^{-1})\|_{t^{-1/2} L^\infty_t} \les 1+\epsilon^{1/2}+r^{-1/2}.
$$
In a similar manner we obtain that
\be\lb{patru}
\|\epsilon \widehat h(\epsilon t) \ast L(t)\|_{t^{-1/2} (1+\log_+|t|)^{-k} L^\infty_t} \les 1 + \epsilon^{1/2} (1+\log_+^k r) + r^{-1/2} + r^{-1/2} \log_+^k r.
\ee

Although it is also possible to carry out the low energy analysis in an $L^2$ setting, as in the proof of Lemma \ref{technical_lemma}, we first do the proof in a weighted $L^1$ setting for a sharper result. Recall the notations $G_0=R_0(0)$, $v=|V|^{1/2}$, $U=\sgn V$, $P=\frac{v \otimes v}{\|V\|_{L^1_x}}$, $Q=I-P$. We write $I+VR_0((\lambda+i0)^2) = U(U+|V|R_0((\lambda+i0)^2))$ and
and
$$\begin{aligned}
U+|V|R_0((\lambda+i0)^2) &= \begin{pmatrix}
Q_1(U+|V|R_0((\lambda+i0)^2)) Q_1 & Q_1(U+|V|R_0((\lambda+i0)^2)) P_1 \\ P_1(U+|V|R_0((\lambda+i0)^2)) Q_1 & P_1(U+|V|R_0((\lambda+i0)^2)) P_1
\end{pmatrix} \\
&=: \begin{pmatrix} \tilde L_{00}(\lambda) & \tilde L_{01}(\lambda) \\ \tilde L_{10}(\lambda) & \tilde L_{11}(\lambda) \end{pmatrix}.
\end{aligned}$$
Here we define the projections $P_1$ and $Q_1$ by $P_1:=\frac{|V| \otimes 1}{\|V\|_{L^1_x}}$ and $Q_1=I-P_1$. They are bounded on $L^1$ and on weighted $L^1$.

The inverse is given for small $\lambda$ by the Feshbach formula (Lemma \ref{fehsbach}):
$$
(U+|V|R_0((\lambda+i0)^2))^{-1} = \begin{pmatrix} \tilde L_{00}^{-1} + \tilde L_{00}^{-1} \tilde L_{01} \tilde C^{-1} \tilde L_{10} \tilde L_{00}^{-1} & -\tilde L_{00}^{-1} \tilde L_{01} \tilde C^{-1} \\ -\tilde C^{-1} \tilde L_{10} \tilde L_{00}^{-1} & \tilde C^{-1} \end{pmatrix},
$$
where $\tilde C=\tilde L_{11}-\tilde L_{10}\tilde L_{00}^{-1}\tilde L_{01}$.

Define the spaces of kernels $\tilde \U_1 := \U_{(1+\log_+|x|)^{-1}L^1_x, (1+\log_+|x|)^{-2} L^1_x \cap \tilde {\mc K}_{1/2}}$ and
\be\begin{aligned}
\tilde {\mc V}_1 := \{T(t, y, x) \mid \|T(t, y, x)\|_{t^{-1/2}(1+ \log_+|t|)^{-2}L^\infty_t} \in \\ \B((1+\log_+|x|)^{-2}L^1_x, (1+\log_+|x|)^{-1}L^1_x \cap \mc K_1)\}.
\end{aligned}\ee
Note that $T \in \tilde {\mc V}_1$ implies that $\|T(t)\|_{\B((1+\log_+|x|)^{-1}L^1_x)} \les t^{-1/2} (1+\log_+|t|)^{-2}$. In addition, $\|T_1 \ast T_2\|_{\tilde \U_1} \les \|T_1\|_{\tilde \U_1} \|T_2\|_{\tilde \U_1}$. From
$$
\|f_1 \ast f_2\|_{t^{-1/2}(1+\log_+|t|)^{-2} L^\infty_t} \les \|f_1\|_{t^{-1/2}(1+\log_+|t|)^{-2}L^\infty_t} \|f_2\|_{L^1_t} + \|f_1\|_{L^1_t} \|f_2\|_{t^{-1/2}(1+\log_+|t|)^{-2}L^\infty_t}
$$
it follows that
$$
\|T_1 \ast T_2\|_{\tilde {\mc V}_1} \les \|T_1\|_{\tilde \U_1} \|T_2\|_{\tilde {\mc V}_1} + \|T_1\|_{\tilde {\mc V}_1} \|T_2\|_{\tilde \U_1}.
$$
More generally
\be\lb{subalgebra}
\bigg\|\prod_{k=1}^n T_k\bigg\|_{\tilde {\mc V}_1} \les C^n \prod_{k=1}^n \|T_k\|_{\tilde \U_1} \cdot \sum_{k=1}^n \frac {\|T_k\|_{\tilde {\mc V}_1}}{\|T_k\|_{\tilde \U_1}}.
\ee

In the Fehsbach formula we first examine $\tilde L_{00}^{-1}$. As in the medium frequency case, let $S_\epsilon(\lambda) = h(\frac \lambda \epsilon)(\tilde L_{00}(\lambda)-\tilde L_{00}(0)$. Due to (\ref{cond}), $\lim_{\epsilon \to 0} \|S_\epsilon\|_{\widehat{\tilde \U_1}} = 0$; this requires that $V(x) \in (1+\log_+|x|)^{-3} L^1_x \cap (1+\log_+|x|)^{-1} \tilde {\mc K}_{1/2} \cap \tilde {\mc K}$.

Zero being a regular point of the spectrum implies that $\tilde L_{00}(0) = Q_1(U+|V| G_0) Q_1$ is invertible on $Q_1 (1+\log_+|x|)^{-1} L^1_x$. Indeed, we consider two cases, according to whether $\int V(x) \dd x = 0$.

If $\int V(x) \dd x \ne 0$ consider the operator $Q_1 U$ on $Q_1 L^1_x = \{f \in L^1_x \mid \int f(x) \dd x = 0\}$. $Q_1 U$ is a compact perturbation of $U$. Therefore, if $Q_1 U$ is not invertible on $Q_1 L^1_x$ there must exist $f \in Q_1 L^1_x$ such that $Q_1 U f = 0$, so $Uf=\alpha |V|$ or in other words $f = \alpha V$. Thus $Q_1 U Q_1$ is invertible on $Q_1 L^1_x$ if and only if $\int V(x) \dd x \ne 0$.

If $\int V(x) \dd x = 0$, choose a compact integral kernel $K(x, y)\in \B(L^1_y, L^\infty_x)$ such that $\int V(x) K(x, y) V(y) \dd x \dd y \ne 0$ and consider the operator $Q_1(U-\delta |V| K)$ on $Q_1 L^1_x$. This is also a compact perturbation of $U$. Therefore, if it is not invertible on $Q_1 L^1_x$ there must exist $f \in Q_1 L^1_x$ such that $Q_1 (U-\delta|V|K) f = 0$, so $(U-\delta|V|K) f = \alpha |V|$, so
$$
f=(U-\delta|V|K)^{-1}|V|=\sum_{n=0}^\infty \delta^n V (K V)^n
$$
for sufficiently small $\delta$. If $\int f(x) \dd x = 0$ for all $\delta$ in a neighborhood of zero, then $\int V (KV)^n \dd x =0$ for all $n$, which is a contradiction when $n=1$.

Therefore there exists some value of $\delta$ (possibly zero) such that $Q_1(U-\delta|V|K)Q_1$ is invertible on $Q_1 L^1_x$. The same goes for weighted $L^1$ and Kato spaces. We then write
$$\begin{aligned}
\tilde L_{00}^{-1}(0) = (Q_1(U+|V|G_0) Q_1)^{-1} = (Q_1 (U-\delta|V|K) Q_1)^{-1} - \\
- v (Q (U + v G_0 v) Q)^{-1} (v G_0 +\delta v K) (Q_1 (U-\delta|V|K) Q_1)^{-1}.
\end{aligned}$$
This shows that $\tilde L_{00}^{-1}(0)$ is bounded on $Q L^1_x$ and on weighted $L^1$, assuming that zero is a regular point of the spectrum. Reiterating
$$
\tilde L_{00}^{-1}(0) = (Q_1 (U-\delta|V|K) Q_1)^{-1} - (Q_1 (U-\delta|V|K) Q_1)^{-1} Q_1 (|V| G_0 + \delta |V| K) Q_1 \tilde L_{00}^{-1}(0),
$$
we also obtain that $\tilde L_{00}^{-1}(0)$ is bounded on the Kato spaces.

We then use the expansion
$$
h(\frac \lambda \epsilon) \tilde L_{00}^{-1}(\lambda) = \sum_{n=0}^\infty (-1)^n \tilde L_{00}^{-1}(0) (S_{2\epsilon}(\lambda) \tilde L_{00}^{-1}(0))^n.
$$

Therefore, for sufficiently small $\epsilon$, $h(\frac \lambda \epsilon) (\tilde L_{00}^{-1}(\lambda) - (Q_1(U+|V|G_0)Q_1)^{-1}) \in \widehat {\tilde \U_1}$ (or equivalently $h(\frac \lambda \epsilon) (\tilde L_{00}^{-1}(\lambda) - U) \in \widehat {\tilde \U_1}$).

In addition, using (\ref{subalgebra}), since by (\ref{patru}) $\|S_\epsilon\|_{\widehat{\tilde {\mc V}_1}} \les 1$ uniformly for small $\epsilon$, we now obtain that $h(\frac \lambda \epsilon) (\tilde L_{00}^{-1}(\lambda)-U) \in \widehat {\tilde {\mc V}_1}$. This requires that $V \in (1+\log_+|x|)^{-3} L^1_x \cap \tilde {\mc K}$.

Next, observe that
$$
\tilde C(\lambda) = (f(\lambda)+h(\lambda)(\frac i 4 \sgn \lambda - \frac 1 {2\pi} \log|\lambda|)) V \otimes 1,
$$
where we knew that $f \in \widehat {L^1}_\lambda$ and we have now obtained that $|\widehat f(t)| \les t^{-1/2}(1+\log_+|t|)^{-2}$. For small $\epsilon$
$$
h(\frac \lambda \epsilon) \tilde C(\lambda)^{-1} = \sum_{n=0}^\infty \frac {(-1)^n h(\frac \lambda \epsilon) f^n(\lambda)}{(\frac i 4 \sgn \lambda - \frac 1 {2\pi} \log |\lambda|)^{n+1}} \frac 1 {\|V\|_{L^1_x}^2} V \otimes 1.
$$
Along the same lines as in the proof of Lemma \ref{technical_lemma} one can show that for $n \geq 1$
$$
\Bigg\|\bigg(\frac {h(\frac \lambda \epsilon)}{(\frac i 4 \sgn \lambda - \frac 1 {2\pi} \log |\lambda|)^n}\bigg)^\wedge\Bigg\|_{t^{-1/2} (1+\log_+|t|)^{-2} L^\infty_t} \les \epsilon^{1/2}\left(\frac 1 {|\log \epsilon|^{n-2}} + \frac {n^2} {|\log\epsilon|^n}\right).
$$
Therefore for small $\epsilon$ $h(\frac \lambda \epsilon) \tilde C(\lambda)^{-1} \in \widehat {\tilde \U_1} \cap \widehat{\tilde {\mc V}_1}$. Examining the other components of $(U+|V| R_0((\lambda+i0)^2))^{-1}$, we see that we have proved that
$$
h(\frac \lambda \epsilon) ((U+|V| R_0((\lambda+i0)^2))^{-1}-U) \in \widehat{\tilde \U_1} \cap \widehat{\tilde {\mc V}_1}
$$
or equivalently
\be\lb{hle}
h(\frac\lambda \epsilon) ((I+VR_0((\lambda+i0)^2))^{-1}-I) \in \widehat{\tilde \U_1} \cap \widehat{\tilde {\mc V}_1}.
\ee

We now consider the Fourier transform of the expression
$$
h(\frac \lambda \epsilon) R_0((\lambda+i0)^2) ((I+VR_0((\lambda+i0)^2))^{-1}-I) V R_0((\lambda+i0)^2).
$$
In this expression, we treat the first and the last factors separately, writing their Fourier transform as
$$
\epsilon \widehat h(\epsilon t) \ast (\one_{t \geq r} ((t^2-r^2)^{-1/2} - t^{-1}) + \one_{t \geq r} t^{-1}),
$$
where the first term is better behaved than the second. Since
$$
\int_{r_1}^{t-r_2} \frac {ds} {s(t-s)} = \frac{\log(t-r_1)+\log(t-r_2)-\log r_1 - \log r_2} t,
$$
we obtain that
\be\lb{dublu}
\|(\one_{t \geq r_1} t^{-1}) \ast (\one_{t \geq r_2}) t^{-1}\|_{t^{-1/2} L^\infty_t} \les 1+ \frac {\log_-\min(r_1, r_2)}{\max(r_1, r_2)^{1/2}} \les 1+ \frac{\log_-r_1}{r_2^{1/2}} + \frac {\log_-r_2}{r_1^{1/2}}.
\ee
Also note that
$$
\int_{r_1+r_2}^t (\one_{t \geq r_1} t^{-1}) \ast (\one_{t \geq r_2} t^{-1}) \dd t \les (\log_+t + \log_-r_1)(\log_+t+\log_-r_2).
$$
Then
$$\begin{aligned}
&\|(\one_{t \geq r_1} t^{-1}) \ast f(t) \ast (\one_{t \geq r_2} t^{-1})\|_{t^{-1/2}L^\infty_t} \les \\
&\les \left(1+ \frac{\log_-r_1}{r_2^{1/2}} + \frac {\log_-r_2}{r_1^{1/2}}\right) \|f\|_{L^1_t} + \\
&+ \|\sup_{s \in [r_1+r_2, t/2]} (\log_+t+\log_-r_1)(\log_+t+\log_-r_2) f(t-s) \dd s\|_{t^{-1/2}L^\infty_t} \\
&\les \left(1+ \frac{\log_-r_1}{r_2^{1/2}} + \frac {\log_-r_2}{r_1^{1/2}}\right) \|f\|_{L^1_t} + \|f\|_{t^{-1/2} (1+\log_+|t|)^{-2}L^\infty_t} + \\
&+ (\log_-r_1 + \log_-r_2) \|f\|_{t^{-1/2} (1+\log_+|t|)^{-1}L^\infty_t} + \log_- r_1 \log_-r_2 \|f\|_{t^{-1/2} L^\infty_t}.
\end{aligned}$$
Taking into account (\ref{hle}) and the definitions of $\tilde \U_1$ and $\tilde {\mc V}_1$, it follows that
\be\lb{restul_rezultat}\begin{aligned}
\left\|\frac 1 {\pi i} \int_{-\infty}^\infty e^{it\lambda} R_0((\lambda+i0)^2) ((I+VR_0((\lambda+i0)^2))^{-1}-I) \right.\\
V R_0((\lambda+i0)^2) h(\frac \lambda \epsilon) f \dd \lambda\|_{L^\infty_x} \les t^{-1/2} \|f\|_{L^1_x}.
\end{aligned}\ee

We next treat separately the term
\be\lb{rvr_low}
\frac 1 {\pi i} \int_{-\infty}^\infty \sin(t\lambda) R_0((\lambda+i0)^2) V R_0((\lambda+i0)^2) h(\frac \lambda \epsilon) \dd \lambda.
\ee
Taking the Fourier transform and denoting $|x-y|=r_1$, $|y-z|=r_2$, we obtain an expression of the form
$$
\int (\epsilon \widehat h(\epsilon t) \ast \one_{t \geq r_1} (t^2-r_1^2)^{-1/2})) V(y) \ast (\epsilon \widehat h(\epsilon t) \ast \one_{t \geq r_2} (t^2-r_2^2)^{-1/2}) \dd y.
$$
The worst-behaved term in this expression is $(\one_{t \geq r_1} t^{-1}) \ast (\one_{t \geq r_2} t^{-1})$, for which we have the bound (\ref{dublu}). Therefore, for $V \in \tilde {\mc K}$,
\be\lb{rvr_rezultat}
(\ref{rvr_low}) \les t^{-1/2} \int |V(y)| (1+\log_-r_1 r_2^{-1/2}+ r_1^{-1/2} \log_-r_2) \les t^{-1/2}.
\ee
From (\ref{restul_rezultat}) and (\ref{rvr_rezultat}) we obtain (\ref{low_freq}).

Combining the results for high, medium, and low frequencies, by means of a partition of unity, we see that we have proved that
$$
\bigg\|\bigg(\frac{\sin(t\sqrt H)P_c}{\sqrt H} - \frac{\sin(t\sqrt{-\Delta})}{\sqrt{-\Delta}}\bigg) f\bigg\|_{L^\infty_x} \les t^{-1/2} \|f\|_{L^1_x}
$$
and likewise for the cosine. Taking into account the well-known results in the free case, see \cite{beals}, the conclusion (\ref{H1}) follows. For (\ref{L^1}), we need to additionally prove that $\langle H \rangle^\alpha \langle \Delta \rangle^{-\beta}$ is $L^1$-bounded for $\alpha<\beta$. This is shown in Appendix C.
\end{proof}

\appendix
\section{The Fourier transforms of Hankel functions}

According to formula (4.7.19) in \cite{andrews}, for $\re \rho>0$
$$
H_0^+(\rho) = \frac 1 {\pi i} \int_{1+i\infty}^{(1+)} e^{i\rho t} (t^2-1)^{-1/2} dt,
$$
where the integral is taken along a contour that starts at $1+i\infty$ with $\arg(t^2-1)=-\pi$, surrounds the point $1$ in counterclockwise fashion, and goes back to $1+i\infty$.

Following Cauchy's theorem, for $\im \rho > 0$ one can change the contour to one that follows the real axis, starting at $+\infty$ with $\arg(t^2-1)=-2\pi$, surrounds the point $1$ in counterclockwise fashion, and goes back to $+\infty$. Letting the contour approach the real axis and $\im \rho \to 0$, we obtain (\ref{hankel1}).

One can make a similar derivation for (\ref{hankel2}), starting from formula (4.7.20) in \cite{andrews}: for $\re\rho>0$
$$
H_0^-(\rho) = \frac 1 {\pi i} \int_{-1+i\infty}^{(-1-)} e^{i\rho t} (t^2-1)^{-1/2} dt,
$$
where the integral is taken along a contour that starts at $-1+i\infty$ with $\arg(t^2-1)=\pi$, surrounds the point $-1$ in the clockwise direction, and goes back to $-1+i\infty$.

\section{The cosine evolution in two dimensions}
We start our derivation from the usual formula for the sine evolution: for $t>0$
$$
(\frac{\sin(t\sqrt{-\Delta})}{\sqrt{-\Delta}} f)(x) = \frac 1 {2\pi} \int_{|y| < t} \frac {f(x-y)}{\sqrt {t^2-|y|^2}} \dd y = \frac 1 {2\pi} \int_0^t \int_{S^1} \frac {f(x+r\omega) r}{\sqrt{t^2-r^2}} \dd \omega \dd r.
$$
Integrating by parts we obtain
$$
(\frac{\sin(t\sqrt{-\Delta})}{\sqrt{-\Delta}} f)(x) = tf(x)+\frac 1 {2\pi}\int_0^t\int_{S^1} \omega\nabla f(x+r\omega) \sqrt{t^2-r^2} \dd \omega \dd r.
$$
Taking a derivative in $t$ we get
$$\begin{aligned}
(\cos(t\sqrt{-\Delta})f)(x) &= f(x)+\frac 1 {2\pi} \int_0^t\int_{S^1} \frac{\omega\nabla f(x+r\omega) t}{\sqrt{t^2-r^2}} \dd \omega \dd r \\
&= f(x)+\frac 1 {2\pi} \int_{|y|<t} \frac {\partial_r f(x+y) t}{\sqrt{t^2-|y|^2}} \dd y,
\end{aligned}$$
where $\partial_r$ is the derivative in the direction of $y$, $\frac {y}{|y|} \nabla$. Taking into account the fact that $f(x) + \int_0^t \omega f(x+r\omega) \dd r = f(x+t\omega)$, we can represent the cosine evolution as
$$\begin{aligned}
(\cos(t\sqrt{-\Delta})f)(x) &= \frac 1 {2\pi} \int_{S^1} f(x+t\omega) \dd \omega + \int_0^t\int_{S^1} \frac{\omega \nabla f(x+r\omega)r^2 \dd\omega\dd r}{(t+\sqrt{t^2-r^2})\sqrt{t^2-r^2}} \\
&= \frac 1 {2\pi t} \int_{|y|=t} f(x+y) \dd y + \frac 1 {2\pi} \int_{|y|<t} \frac {\partial_r f(x+y) |y| \dd y}{(t+\sqrt{t^2-|y|^2})\sqrt{t^2-|y|^2}}.
\end{aligned}$$
Assuming that $f$ goes to zero at infinity, then $f(x+t\omega) = -\int_t^\infty \omega \nabla f(x+r\omega) \dd r$. Therefore
\be\lb{est_cos}\begin{aligned}
(\cos(t\sqrt{-\Delta})f)(x) &= -\frac 1 {2\pi} \int_t^\infty \int_{S^1} \omega \nabla f(x+r\omega) \dd \omega + \int_0^t\int_{S^1} \frac{f(x+r\omega)r^2 \dd \omega \dd r}{(t+\sqrt{t^2-r^2})\sqrt{t^2-r^2}} \\
&= -\frac 1 {2\pi} \int_{|y|\geq t} \frac {\partial_r f(x+y) \dd y}{|y|} + \frac 1 {2\pi} \int_{|y|<t} \frac {\partial_r f(x+y) |y| \dd y}{(t+\sqrt{t^2-|y|^2})\sqrt{t^2-|y|^2}}.
\end{aligned}\ee

\section{Boundedness of some Fourier multipliers}
\begin{lemma} Assume that $V \in L^q$ for some $q>1$. For $0<\alpha<\beta<1$
$$
\|\langle H \rangle^\alpha (-\Delta+1)^{-\beta} f\|_{L^1_x} \les \|f\|_{L^1_x}.
$$
\end{lemma}
\begin{proof} For simplicity, we represent $\langle H \rangle$ as $H+\lambda_0$ for some sufficiently large $\lambda_0$. We make use of the following representation formula: for $0<p<1$
$$
A^{-p} = C_p \int_0^\infty (A+\lambda)^{-1} \lambda^{-p} \dd \lambda
$$
and
$$
A^{1-p} = C_p \int_0^\infty A(A+\lambda)^{-1} \lambda^{-p} \dd \lambda.
$$
Therefore
\be\lb{double}
\langle H \rangle^\alpha (-\Delta+1)^{-\beta} = C_\alpha C_\beta \int_0^\infty \int_0^\infty \lambda^{\alpha-1} \mu^{-\beta} (H+\lambda_0)(H+\lambda_0+\lambda)^{-1} (-\Delta+1+\mu)^{-1} \dd \lambda \dd \mu.
\ee
Note that $\|(-\Delta+1+\mu)^{-1}\|_{\B(L^1_x)} \les (1+\mu)^{-1}$ (easy to prove using duality and $L^\infty$),
$$
\|\Delta(-\Delta+1+\mu)^{-1}\|_{\B(L^1)} = \|-I+(1+\mu)(-\Delta+1+\mu)^{-1}\|_{\B(L^1_x)} \les 1,
$$
$$\begin{aligned}
\|V(-\Delta+1+\mu)^{-1}\|_{\B(L^1)} &\les \|V\|_{L^q_x} \|(-\Delta+1+\mu)^{-1}\|_{\B(L^1_x, L^{q'}_x)} \\
&\les \|V\|_{L^q_x} (1+\mu)^{\frac 1 q - 1} (q-1)^{-\frac 1 q} \les 1,
\end{aligned}$$
so $\|(H+\lambda_0)(-\Delta+1+\mu)^{-1}\|_{\B(L^1_x)} \les 1$. Similarly we see that for sufficiently large $\lambda_0$ and all $\lambda\geq 0$
$$\begin{aligned}
\|(H+\lambda_0+\lambda)^{-1}\|_{\B(L^1_x)} &\les \|(-\Delta+\lambda_0+\lambda)^{-1}\|_{\B(L^1_x)} \|(I+ V (-\Delta+\lambda_0+\lambda)^{-1})^{-1}\|_{\B(L^1_x)} \\
&\les (\lambda_0+\lambda)^{-1}
\end{aligned}$$
because $\|V (-\Delta+\lambda_0+\lambda)^{-1})^{-1}\|_{\B(L^1)} \leq 1/2$. Therefore 
$$
\|(H+\lambda_0)(H+\lambda_0+\lambda)^{-1}\|_{\B(L^1_x)} = \|I - \lambda(H+\lambda_0+\lambda)^{-1}\|_{\B(L^1_x)} \les 1.
$$
In conclusion
$$\begin{aligned}
\|(H+\lambda_0)(H+\lambda_0+\lambda)^{-1} (-\Delta+1+\mu)^{-1}\|_{\B(L^1_x)} &\les \min((\lambda_0+\lambda)^{-1}, (1+\mu)^{-1}) \\
&\les (1+\lambda)^{-\frac {\alpha+\beta} 2} (1+\mu)^{\frac {\alpha+\beta} 2 - 1}.
\end{aligned}$$
Plugging this back into (\ref{double}), we obtain that
$$
\|\langle H \rangle^\alpha (-\Delta+1)^{-\beta}\|_{\B(L^1)} \les \int_0^\infty \int_0^\infty \lambda^{\alpha-1} (1+\lambda)^{-\frac {\alpha+\beta} 2} \mu^{-\beta} (1+\mu)^{\frac {\alpha+\beta} 2 - 1} \dd \lambda \dd \mu < \infty.
$$
\end{proof}

\section*{Acknowledgement}
I would like to thank Sung-Jin Oh for the interesting discussion.

\end{document}